\documentclass[11pt,reqno,a4paper]{amsart}
\usepackage{hyperref}
\usepackage{latexsym}
\usepackage{amssymb,amsmath}
\usepackage{graphics}
\usepackage{url}
\usepackage{bbm}
\usepackage{multirow}
\usepackage{booktabs}  
\usepackage{enumerate}
\usepackage{bbm}
\usepackage[usenames,dvipsnames]{color}
\usepackage{dsfont}
\allowdisplaybreaks

\usepackage{tikz}
\usetikzlibrary{matrix,chains,shapes,arrows,scopes,positioning}
\usetikzlibrary{calc,decorations.pathmorphing,shapes}
\usetikzlibrary{patterns}
\tikzstyle{empty}=[circle,draw=black!80,thick]
\tikzstyle{emptyn}=[circle,draw=black!80,fill=white,scale=0.5] 
\tikzstyle{nero}=[circle,draw=black!80,fill=black!80,thick] 

\newcommand{\fS}{\mathfrak{S}}
\newcommand{\Irr}{{\operatorname{Irr}}}
\newcommand{\Syl}{\operatorname{Syl}}
\newcommand{\Lin}{\operatorname{Lin}}
\newcommand{\Char}{\operatorname{Char}}
\newcommand{\triv}{\mathbbm{1}}
\newcommand{\down}{\big\downarrow}
\newcommand{\up}{\big\uparrow}
\newcommand{\ccl}{\operatorname{ccl}}
\newcommand{\partition}{\operatorname{Part}}
\newcommand{\cut}{\operatorname{Cut}}
\newcommand{\Sum}{\textstyle\sum}

\newcommand{\itemspace}{\setlength\itemsep{5pt}}

\newtheorem{theorem}{Theorem}[section]
\newtheorem{lemma}[theorem]{Lemma}

\newtheorem{corollary}[theorem]{Corollary}
\newtheorem{proposition}[theorem]{Proposition}
\newtheorem{definition}[theorem]{Definition}

\theoremstyle{definition}
\newtheorem{example}[theorem]{Example}

\newtheorem{notation}[theorem]{Notation}
\newtheorem{remark}[theorem]{Remark}

\newtheorem{theoremA}{Theorem}

\makeatletter
\@namedef{subjclassname@2020}{%
	\textup{2020} Mathematics Subject Classification}
\makeatother

\begin{document}

\title[Linear characters of Sylow subgroups of $\fS_n$]{Linear characters of Sylow subgroups of symmetric groups}


\author{Eugenio Giannelli}
\address[E. Giannelli]{Dipartimento di Matematica e Informatica	U.~Dini, Viale Morgagni 67/a, Firenze, Italy}
\email{eugenio.giannelli@unifi.it}

\author{Stacey Law}
\address[S. Law]{Department of Pure Mathematics and Mathematical Statistics, University of Cambridge, Cambridge CB3 0WB, UK}
\email{swcl2@cam.ac.uk}

\author{Jason Long}
\address[J. Long]{Mathematical Institute, University of Oxford, Radcliffe Observatory, Andrew Wiles Building, Woodstock Rd, Oxford OX2 6GG, UK}
\email{jason.long@maths.ox.ac.uk}


\begin{abstract}
Let $p$ be any prime. Let $P_n$ be a Sylow $p$-subgroup of the symmetric group $\fS_n$. Let $\phi$ and $\psi$ be linear characters of $P_n$ and let $N$ be the normaliser of $P_n$ in $\fS_n$. In this article we show that the inductions of $\phi$ and $\psi$ to $\fS_n$ are equal if, and only if, $\phi$ and $\psi$ are $N$--conjugate. This is an analogue for symmetric groups of a result of Navarro for $p$-solvable groups.
\end{abstract}



\maketitle

\section{Introduction}\label{sec:intro}

Let $G$ be a finite group and let $P$ be a Sylow subgroup of $G$.
The algebraic properties of those characters of $G$ obtained by inducing irreducible characters of $P$ encodes interesting information about the group $G$ itself \cite{NavSylSurvey}. For example, the normality of $P$ in $G$ is controlled by specific properties of the permutation character $\triv_{P}\big\uparrow^G$. More precisely, in \cite{MN12} it is proved that $P$ is normal in $G$ if, and only if, all the irreducible constituents of $\triv_{P}\big\uparrow^G$ have degree coprime to $p$. On the other hand, in \cite{NTV} it is shown that $P$ is self-normalising if, and only if, $\triv_G$ is the only constituent of $\triv_{P}\big\uparrow^G$ having degree coprime to $p$. 

Recently, the first two authors have studied the induction of characters of $P$ to $G$ in the case where $G=\fS_n$ is a symmetric group. The irreducible constituents of $\triv_{P}\big\uparrow^{\fS_n}$ were completely described in \cite{GL1}. Later in \cite{GL2}, the focus was extended to the study of any monomial character $\phi\big\uparrow^{\fS_n}$, where $\phi$ is an arbitrary linear character of $P\in\Syl_{p}(\fS_n)$. 
The present article complements the study begun in \cite{GL2} by extending to symmetric groups a result proved by Navarro for $p$-solvable groups \cite[Theorem C]{N}. 

Let $p$ be a prime. For any finite group $G$ and $P$ a Sylow $p$-subgroup of $G$, the normaliser $N=N_G(P)$ acts by conjugation on the set $\Lin(P)$ of linear characters of $P$. It is easy to see that if two characters are $N$--conjugate then their inductions to $N$, and hence $G$, are equal. Therefore, in order to study structural properties of such induced characters $\phi\up^G$, it is sufficient to consider only a set of representatives for the orbits of $N$ in its action on $\Lin(P)$. However, is the converse true? That is, if $\phi$ and $\psi$ are two linear characters of $P$ such that $\phi\up^G=\psi\up^G$, must $\phi$ and $\psi$ be $N$--conjugate?
As mentioned above, this was answered in the affirmative for all $p$-solvable groups by Navarro in \cite{N}, though there exist finite groups (such as $\text{PSL}(3,3)$ with $p=3$) for which the answer is negative. 
Our main result shows that the answer to this question is also affirmative for symmetric groups at all primes $p$.

\begin{theoremA}\label{thm:A}
Let $n\in\mathbb{N}$ and $p$ be any prime. Let $P_n\in\Syl_p(\fS_n)$ and let $N=N_{\fS_n}(P_n)$. Let $\phi$ and $\psi$ be linear characters of $P_n$. Then $\phi\up^{\fS_n}_{P_n}=\psi\up^{\fS_n}_{P_n}$ if, and only if, $\phi$ and $\psi$ are $N$--conjugate.
\end{theoremA}

In order to prove Theorem~\ref{thm:A}, we calculate the values of the induced characters $\phi\up^{\fS_n}$ and $\psi\up^{\fS_n}$ on a range of different conjugacy classes of $S_n$. By relating these values to combinatorial properties of $\phi$ and $\psi$, we demonstrate that if the induced characters agree on sufficiently many conjugacy classes, then in fact the linear characters $\phi$ and $\psi$ are $N$--conjugate.
We believe that the closed formulae we provide in Lemma \ref{lem: b=1}, Proposition \ref{prop: b>1} and Theorem \ref{thm:general-a} may be useful for future investigations, and could be of independent interest. 
	
\subsection{Structure of the article}

We begin with a reduction of Theorem~\ref{thm:A} to the case where $n=ap^k$ for some $a\in\{1,2,\dots,p-1\}$ and $k\in\mathbb{N}$. This is stated separately as Theorem~\ref{thm: apk}. In Section~\ref{sec:reduction}, we prove Theorem~\ref{thm:A} assuming Theorem~\ref{thm: apk}.

In Section~\ref{sec:newnotation} we 
describe a certain parametrisation $\phi=\phi(\underline{\mathbf{u}})$ of the linear characters of $P_n$ and formulate what it means for linear characters to be $N$--conjugate in terms of this parametrisation. We then state Theorem~\ref{thm: apk restated} which re-expresses Theorem~\ref{thm: apk} using this new notation, and whose proof will be the main aim of the remaining sections of the paper.

Section~\ref{sec:char values} represents the technical core of the argument. Here, we tackle the computation of induced character values, culminating in Theorem~\ref{thm:general-a}. While it will quickly become clear that induced character $\phi\up^{\fS_n}$ can be calculated in a recursive fashion, the main difficulty will be in simplifying the resulting expressions so that they can be used to analyse the parametrisation $\phi(\underline{\mathbf{u}})$. To achieve this, we make use of a number of combinatorial tools, including identities concerning Bell polynomials and Stirling numbers of the second kind.

In Section~\ref{sec:final} we shall use the induced character formulae from Section~\ref{sec:char values} to establish Theorem~\ref{thm: apk restated}, and hence Theorem~\ref{thm:A}.

In Appendix~\ref{sec:appendix}, we provide detailed examples to clarify some of the expressions in Section~\ref{sec:char values}. 
The paper is self-contained and the examples are not necessary to follow the proof, but the reader may find it helpful to refer to the appendix when reading the proof of Proposition~\ref{prop: b>1}.

\subsection*{Acknowledgements}
We thank the anonymous reviewer for his/her helpful corrections.
The second author was supported by a London Mathematical Society Early Career Fellowship at the University of Oxford (ECF-1819-05).

\bigskip

\section{Reduction to $n=ap^k$}\label{sec:reduction}

We begin by introducing some notation used throughout the article.
We refer the reader to \cite{JK} for a detailed account of the representation theory of symmetric groups. 

For a finite group $G$, let $\Char(G)$ denote the set of ordinary characters of $G$. Let $\Irr(G)$ the subset of those which are irreducible, and $\Lin(G)$ for those which are linear, i.e.~of degree 1. 
For $m$ a natural number, let $[m]$ denote the set $\{1,2,\dotsc,m\}$. 
The symmetric group on the set $[n]=\{1,2,\dotsc,n\}$ will be denoted by $\fS_n$ throughout. 

\smallskip

First, we record a proof of the \textit{easy} direction of Theorem A.

\begin{lemma}\label{lem:easy direction}
	Let $G$ be a finite group and $p$ be a prime. Let $P\in\Syl_p(G)$ and $N=N_G(P)$. Suppose $\phi,\psi\in\Char(P)$ and $\psi=\phi^m$ for some $m\in N$. Then $\phi\up^G_P=\psi\up^G_P$.
\end{lemma}

\begin{proof}
	Let $\alpha\in\Irr(N)$. Note that $\alpha^m=\alpha$. Then by Frobenius reciprocity,
	$$\langle \phi\up^N,\alpha\rangle = \langle \phi, \alpha\down_P\rangle = \langle \phi^m, (\alpha\down_P)^m \rangle = \langle \psi, (\alpha^m)\down_P\rangle = \langle \psi\up^N, \alpha\rangle.$$
	Thus $\phi\up^N=\psi\up^N$, as $\alpha$ is arbitrary. It follows that $\phi\up^G_P=\psi\up^G_P$.
\end{proof}

The \textit{difficult} part of Theorem~\ref{thm:A} is the converse of Lemma~\ref{lem:easy direction}. It turns out that this converse reduces to the case when $n=ap^k$, which we now state.

\begin{theorem}\label{thm: apk}
	Let $a\in [p-1]$ and $k\in\mathbb{N}$. Let $\phi,\psi\in\mathrm{Lin}(P_{ap^k})$ be such that $\phi\up^{\fS_{ap^k}}=\psi\up^{\fS_{ap^k}}$. Then $\phi$ and $\psi$ are $N_{\fS_{ap^k}}(P_{ap^k})$--conjugate. 
\end{theorem}

\noindent Theorem~\ref{thm: apk} is proved in Section~\ref{sec:final}, following explicit calculations of certain induced character values in Section~\ref{sec:char values}.
In the rest of this section, we deduce Theorem~\ref{thm:A} from Theorem~\ref{thm: apk}.

\medskip

Recall that for each $n\in\mathbb{N}$, $\Irr(\fS_n)$ is naturally in bijection with the set of all partitions of $n$. For a partition $\lambda$ of the number $n$, also written $\lambda\vdash n$, we denote the corresponding irreducible character of $\fS_n$ by $\chi^\lambda$. 

\begin{lemma}\label{lem:C3}
	Let $n,m\in\mathbb{N}$. Let $A,B\in\Char(\fS_n)$ and $Z\in\Char(\fS_m)$. Then 
	$$(A\times Z)\big\uparrow^{\fS_{n+m}}_{\fS_n\times\fS_m}=(B\times Z)\big\uparrow^{\fS_{n+m}}_{\fS_n\times\fS_m}\ \ \text{if and only if}\ \ A=B.$$ 
\end{lemma}

\begin{proof}
	The `if' direction is clear, so now suppose that $(A\times Z)\big\uparrow^{\fS_{n+m}}=(B\times Z)\big\uparrow^{\fS_{n+m}}$ and assume for a contradiction that $A\neq B$. 
	For $X\in\{A,B,Z\}$, we let $c^X_\lambda:=\left\langle X, \chi^\lambda\right\rangle$, where $\lambda$ is a partition of $n$ (respectively $m$) if $X\in\{A,B\}$ (respectively $X=Z$).
	We define the following sets:
	$$\mathcal{M}=\{\lambda\vdash n\ |\ c_\lambda^A\neq c_\lambda^B\}\ \ \text{and}\ \ \mathcal{N}=\{\mu\vdash m\ |\ c_\mu^Z\neq 0\},$$
	which by assumption are non-empty.
	Let $\overline{\lambda}$ and $\overline{\mu}$ be the lexicographically greatest partitions in $\mathcal{M}$ and $\mathcal{N}$ respectively, and let $\alpha$ be the partition of $n+m$ defined by $\alpha=\overline{\lambda}+\overline{\mu}:=(\lambda_1+\mu_1,\lambda_2+\mu_2,\ldots )$.
	By the Littlewood\textendash Richardson rule \cite[Chapter 16]{J}, we have that 
	\begin{align*}
	\left\langle(A\times Z)\uparrow^{\fS_{n+m}}, \chi^\alpha\right\rangle &= c_{\overline{\lambda}}^A c_{\overline{\mu}}^Z c_{\overline{\lambda}\overline{\mu}}^{\alpha} + \sum_{\lambda>\overline{\lambda}} \sum_{\mu\in \mathcal{N}} c_{\lambda}^A c_{\mu}^Z c_{\lambda\mu}^{\alpha} = c_{\overline{\lambda}}^A c_{\overline{\mu}}^Z c_{\overline{\lambda}\overline{\mu}}^{\alpha} + \sum_{\lambda>\overline{\lambda}} \sum_{\mu\in \mathcal{N}} c_{\lambda}^B c_{\mu}^Z c_{\lambda\mu}^{\alpha}\\ 
	&\neq c_{\overline{\lambda}}^B c_{\overline{\mu}}^Z c_{\overline{\lambda}\overline{\mu}}^{\alpha} + \sum_{\lambda>\overline{\lambda}} \sum_{\mu\in \mathcal{N}} c_{\lambda}^B c_{\mu}^Z c_{\lambda\mu}^{\alpha} = \left\langle(B\times Z)\uparrow^{\fS_{n+m}}, \chi^\alpha\right\rangle.
	\end{align*}
	This is a contradiction, and the proof is concluded. 
\end{proof}

\begin{lemma}\label{lem: C6}
	Let $b, n, m\in \mathbb{N}$ with $n>m$. Let $P\times Q\leq \fS_{bn}\times\fS_m\leq \fS_{bn+m}$ be such that $P$ contains an element $\sigma$ 
	which is a product of $b$ disjoint $n$-cycles. Let $g\in Q$. Let $\chi$ and $\eta$ be class functions of $P$ and $Q$ respectively. Then 
	$$(\chi\times\eta)\big\uparrow_{P\times Q}^{\fS_{bn+m}}(\sigma g)=\chi\big\uparrow_P^{\fS_{bn}}(\sigma)\cdot\eta\big\uparrow_Q^{\fS_m}(g).$$
\end{lemma}

\begin{proof}
	This follows from the definition of induction, after observing that $\sigma^x\in P\times Q$ if and only if $\sigma^x\in P$, for all $x\in\fS_{bn+m}$. 
\end{proof}

Assuming Theorem~\ref{thm: apk}, we now prove Theorem A.

\begin{proof}[Proof of Theorem A]
	Recall that Lemma~\ref{lem:easy direction} gives the easy direction of Theorem A. We now prove the difficult direction. Let $n=a_1p^{n_1}+\cdots+a_tp^{n_t}$ be the $p$-adic expansion of $n$, with $t\in\mathbb{N}$, $a_i\in[p-1]$ for all $i$ and $0\le n_1<\cdots<n_t$. Suppose $\phi\up^{\fS_n}_{P_n}=\psi\up^{\fS_n}_{P_n}$. We wish to show that $\phi$ and $\psi$ are $N_{\fS_n}(P_n)$--conjugate.
	
	We proceed by induction on $t$, noting that Theorem~\ref{thm: apk} gives the case of $t=1$. Now suppose $t\geq 2$ and assume for a contradiction that $\phi$ and $\psi$ are not $N_{\fS_{n}}(P_{n})$--conjugate. Let $m=a_tp^{n_t}$ and write $\phi=\phi_1\times\phi_2$ and $\psi=\psi_1\times\psi_2$ where $\phi_1,\psi_1\in\mathrm{Lin}(P_m)$ and $\phi_2,\psi_2\in\mathrm{Lin}(P_{n-m})$.
	Since $\phi$ and $\psi$ are not $N_{\fS_{n}}(P_{n})$--conjugate and $N_{\fS_n}(P_n)\cong 
	N_{\fS_m}(P_m)\times N_{\fS_{n-m}}(P_{n-m})$ (see \cite[\textsection 2.3.2]{SLThesis}, for instance), at least one of the following two statements must hold: 
	
	\begin{itemize}
		\item[(i)] $\phi_1$ and $\psi_1$ are not $N_{\fS_m}(P_m)$--conjugate; 
		\item[(ii)] $\phi_2$ and $\psi_2$ are not $N_{\fS_{n-m}}(P_{n-m})$--conjugate. 
	\end{itemize}
	Since $P_n\cong P_m\times P_{n-m}$, we have that 
	\[ (\phi_1\up^{\fS_m}_{P_m}\times\phi_2\up^{\fS_{n-m}}_{P_{n-m}})\up^{\fS_n} 
	= \phi\up^{\fS_n}_{P_n} = \psi\up^{\fS_n}_{P_n} = (\psi_1\up^{\fS_m}_{P_m}\times\psi_2\up^{\fS_{n-m}}_{P_{n-m}})\up^{\fS_n} \]
	and so using Lemmas~\ref{lem:easy direction}, \ref{lem:C3} and the inductive hypothesis, we deduce that both conditions (i) and (ii) must hold. 
	Let $g\in \fS_{n-m}$ be such that $\phi_2\up^{\fS_{n-m}}(g)\ne\psi_2\up^{\fS_{n-m}}(g)$; such an element exists by the inductive hypothesis. Let $\sigma\in P_m\le \fS_m$ be a product of $a_t$ disjoint $p^{n_t}$-cycles. We now denote by $h$ the element of $\fS_m\times\fS_{n-m}\le \fS_n$ defined as follows:
	\[ h = \begin{cases} 
	\sigma & \mathrm{if}\ \phi_1\up^{\fS_m}(\sigma)\ne \psi_1\up^{\fS_m}(\sigma),\\
	\sigma g & \mathrm{otherwise}.\end{cases} \]
	We remark that $\phi\up^{\fS_m}(\sigma)\ne 0$ for all $\phi\in\Lin(P_m)$: this follows from Corollary~\ref{cor:non-zero} below. (The proof could be included here, but it will be much shorter to describe following the introduction of relevant notation and ideas later.)
	Then $\phi\up^{\fS_n}(h)\ne \psi\up^{\fS_n}(h)$ by Lemma~\ref{lem: C6}, a contradiction.
\end{proof}

\bigskip

\section{Wreath products and the structure of the Sylow subgroup}\label{sec:newnotation}
We recall some basic facts concerning wreath products and the representation theory of the Sylow subgroups $P_n$ of $\fS_n$ (see \cite[Chapter 4]{JK} for further detail). In Section~\ref{subsec:param} we describe a parametrisation of the linear characters of $P_n$ which will play a central role in our calculations of induced character values.

\subsection{Wreath products}\label{sec:wreath}
Let $G$ be a finite group, $n\in\mathbb{N}$ and $H\le\fS_n$. We denote by $G^{\times n}$ the direct product of $n$ copies of $G$. The natural action of $\fS_n$ on the direct factors of $G^{\times n}$ induces an action of $\fS_n$ (and therefore of $H$) via automorphisms of $G^{\times n}$, giving the wreath product $G\wr H:= G^{\times n}\rtimes H$. 
As in \cite[Chapter 4]{JK}, we denote the elements of $G\wr H$ by $(g_1,\dotsc,g_n;h)$ for $g_i\in G$ and $h\in H$. Let $V$ be a $\mathbb{C}G$--module and suppose it affords the character $\phi$.
We let $V^{\otimes n}:=V\otimes\cdots\otimes V$ be the corresponding $\mathbb{C}G^{\times n}$--module. The left action of $G\wr H$ on $V^{\otimes n}$ defined by linearly extending
\begin{equation*}
(g_1,\dotsc,g_n;h)\ :\quad v_1\otimes \cdots\otimes v_n \longmapsto g_1v_{h^{-1}(1)}\otimes\cdots\otimes g_nv_{h^{-1}(n)}
\end{equation*}
turns $V^{\otimes n}$ into a $\mathbb{C}(G\wr H)$--module, which we denote by $\widetilde{V^{\otimes n}}$ (see \cite[(4.3.7)]{JK}). We denote by $\tilde{\phi}$ the character afforded by the $\mathbb{C}(G\wr H)$--module $\widetilde{V^{\otimes n}}$. For any ordinary character $\psi$ of $H$, we let $\psi$ also denote its inflation to $G\wr H$ and let
$$\mathcal{X}(\phi;\psi):=\tilde{\phi}\cdot\psi$$
be the ordinary character of $G\wr H$ obtained as the 
product of $\tilde{\phi}$ and $\psi$.

We record some basic results that will be useful later.

\begin{lemma}[{\cite[Lemma 4.3.9]{JK}}]\label{lem:4.3.9}
	Let $n\in\mathbb{N}$. Let $H\le\fS_n$ and $G$ be finite groups. Let $\phi\in\Irr(G)$ and $\psi\in\Irr(H)$. Then for all $f_1,\dotsc,f_n\in G$ and $\pi\in H$,
	$$ \mathcal{X}(\phi;\psi)(f_1,\dotsc,f_n;\pi) = \prod_{v=1}^{c(\pi)} \phi( f_{j_v}\cdot f_{\pi^{-1}(j_v)}\cdot f_{\pi^{-2}(j_v)}\cdots f_{\pi^{-l_v+1}(j_v)}) \cdot \psi(\pi),$$
	where $c(\pi)$ is the number of disjoint cycles in $\pi$, $l_v$ is the length of the $v^{th}$ cycle, and for each $v$, $j_v$ is some fixed element in the $v^{th}$ cycle.
\end{lemma}

\noindent For example, if $n=8$ and $\pi=(1,3,7,2)(5,8,6)(4)$, then $\mathcal{X}(\phi;\psi)(f_1,\dotsc,f_8;\pi) = \phi(f_2f_7f_3f_1)\cdot \phi(f_6f_8f_5)\cdot \phi(f_4)\cdot\psi(\pi)$.

\begin{lemma}{\cite[Lemma 2.22]{SLThesis}}\label{lem: pk cycles}
	Let $k\in\mathbb{N}$ and $p$ be a prime. Let $x=(f_1,\dotsc,f_p;\sigma)\in P_{p^k}$ 
	for some $f_i\in P_{p^{k-1}}$ and $\sigma\in P_p$. If $x$ has a fixed point, then $\sigma=1$. Moreover, $x$ is a $p^k$-cycle if and only if $\sigma\ne 1$ and $f_{\sigma^{p-1}(1)}\cdots f_{\sigma(1)}\cdot f_1$ is a $p^{k-1}$-cycle. 
\end{lemma}

\medskip

\subsection{The representation theory of the Sylow subgroups of $\fS_n$}\label{subsec:param}

Fix a prime $p$ and let $n\in\mathbb{N}$. As in \cite[\textsection 2.3.2]{SLThesis}, for every $n$ we fix a Sylow $p$-subgroup $P_n$ of $\fS_n$ such that $P_m\le P_n$ whenever $m$ is a $p$-adic subsum of $n$.
Clearly $P_1$ is trivial while $P_p$ is cyclic of order $p$. More generally, $P_{p^k}= (P_{p^{k-1}})^{\times p}\rtimes P_p=P_{p^{k-1}}\wr P_p\cong P_p\wr \cdots \wr P_p$ ($k$-fold wreath product) for all $k\in\mathbb{N}$.
For an arbitrary natural number $n$, suppose it has $p$-adic expansion $n=\sum_{i=1}^t a_ip^{n_i}$. That is, $t\in\mathbb{N}$, $0\le n_1<\cdots<n_t$ and $a_i\in[p-1]$ for all $i$. Then $P_n\cong (P_{p^{n_1}})^{\times a_1}\times\cdots\times (P_{p^{n_t}})^{\times a_t}$.

\smallskip

We fix a parametrisation of the linear characters of $P_n$ for all $n$ as follows (see \cite[\textsection 2.3.2]{SLThesis} or \cite{GL2} for more detail).
First we consider the case $n=p^k$. When $k=1$, observe that $\Lin(P_p) = \Irr(P_p)$. When $k\ge 2$, \cite[Corollary 6.17]{IBook} gives
\[ \Lin(P_{p^k}) = \{ \mathcal{X}(\phi;\psi) : \phi\in\Lin(P_{p^{k-1}}),\ \psi\in \Lin(P_p) \}. \]
This gives a natural bijection\footnote{We remark that $P_{p^k}/{P_{p^k}}'$ is isomorphic to $(P_p)^k$. Fixing a natural isomorphism $P_{p^k}/{P_{p^k}}'\cong(P_p)^k$, our indexing of $\Lin(P_{p^k})$ may be obtained equivalently from the canonical bijection $\Lin(P_{p^k})\to \Irr(P_{p^k}/{P_{p^k}}')$.} between $\Lin(P_{p^k})$ and $(\Lin(P_p))^k$. More precisely, we may parametrise $\Lin(P_{p^k})$ by writing 
\[ \Lin(P_{p^k}) = \{ \mathcal{X}(\phi_1,\dotsc,\phi_k) : \phi_i\in\Lin(P_p) \}\]
where $\mathcal{X}(\phi_1,\dotsc,\phi_k)$ is defined recursively via $\mathcal{X}(\phi_1)=\phi_1$ when $k=1$, and
\[ \mathcal{X}(\phi_1,\dotsc,\phi_k) := \mathcal{X}\big( \mathcal{X}(\phi_1,\dotsc,\phi_{k-1}); \phi_k \big) \]
when $k\ge 2$.
The following lemma characterises when two linear characters of $P_{p^k}$ are conjugate via an element of $N_{\fS_{p^k}}(P_{p^k})$ in terms of this parametrisation. Here $\triv_{P_p}$ denotes the trivial character of $P_p$. 
\begin{lemma}\label{lem:C1}
	Let $k\in\mathbb{N}$ and let $\theta_1,\theta_2\in\Lin(P_{p^k})$. Suppose $\theta_1=\mathcal{X}(\phi_1,\dotsc,\phi_k)$ and $\theta_2=\mathcal{X}(\psi_1,\dotsc,\psi_k)$. Then $\theta_1$ and $\theta_2$ are $N_{\fS_{p^k}}(P_{p^k})$--conjugate if and only if $\{i : \phi_i=\triv_{P_p}\} = \{j:\psi_j=\triv_{P_p} \}$.
\end{lemma}
\begin{proof}
This follows from a standard argument. We refer the reader to either \cite[Lemma 4.3]{SLThesis} or to \cite[Corollary 2.5]{G20} for a complete proof.
\end{proof}

In other words, $\mathcal{X}(\phi_1,\dotsc,\phi_k)$ and $\mathcal{X}(\psi_1,\dotsc,\psi_k)$ are $N_{\fS_{p^k}}(P_{p^k})$--conjugate if and only if $\triv_{P_p}$ occurs in exactly the same positions in the two sequences $\phi_1,\dotsc,\phi_k$ and $\psi_1,\dotsc,\psi_k$.

Next, we consider arbitrary natural numbers $n$. Suppose $n=\sum_{i=1}^t a_ip^{n_i}$ in $p$-adic expansion, as above. Since $P_n\cong (P_{p^{n_1}})^{\times a_1}\times\cdots\times (P_{p^{n_t}})^{\times a_t}$, we have
\[ \Lin(P_n) = (\Lin(P_{p^{n_1}}))^{a_1} \times \cdots \times (\Lin(P_{p^{n_t}}))^{a_t}. \]
The case of $n=ap^k$ for $k\in\mathbb{N}$ and $a\in[p-1]$ is the most important for our purposes, since our goal is to prove Theorem~\ref{thm: apk}. Therefore, while a statement analogous to Lemma~\ref{lem:C1} holds for all $n$, we need only record it here in the case $n=ap^k$. We note that $\Lin(P_{ap^k})$ may be parametrised as follows:
\[ \Lin(P_{ap^k}) = \{ \mathcal{L}(\phi_{1,1},\phi_{1,2},\dotsc,\phi_{a,k-1},\phi_{a,k}) : \phi_{i,j}\in\Lin(P_p)\ \forall\ i\in[a],\ j\in[k] \} \]
where $\mathcal{L}(\phi_{1,1},\dotsc,\phi_{a,k}) = \mathcal{X}(\phi_{1,1},\dotsc,\phi_{1,k})\times\cdots\times\mathcal{X}(\phi_{a,1},\dotsc,\phi_{a,k})$.

\begin{lemma}\label{lem:C2}
	Let $k\in\mathbb{N}$ and let $a\in[p-1]$. Let $\theta_1,\theta_2\in\Lin(P_{ap^k})$. Suppose $\theta_1=\mathcal{L}(\phi_{1,1},\dotsc,\phi_{a,k})$ and $\theta_2=\mathcal{L}(\psi_{1,1},\dotsc,\psi_{a,k})$. Then $\theta_1$ and $\theta_2$ are $N_{\fS_n}(P_n)$--conjugate if and only if there exists $\sigma\in\fS_a$ such that for each $i\in[a]$, $\mathcal{X}(\phi_{i,1},\dotsc,\phi_{i,k})$ and $\mathcal{X}(\psi_{\sigma(i),1},\dotsc,\psi_{\sigma(i),k})$ are $N_{\fS_{p^k}}(P_{p^k})$--conjugate.
\end{lemma}
\begin{proof}
As with Lemma \ref{lem:C1}, we refer the reader to either \cite[Lemma 4.5]{SLThesis} or to \cite[Lemma 2.14]{G20} for a complete proof.
\end{proof}

In light of Lemmas~\ref{lem:C1} and~\ref{lem:C2} and our goal of proving Theorem~\ref{thm: apk}, we may simplify the notation used to parametrise the linear characters of $P_{ap^k}$.

\begin{notation}
	Let $k\in\mathbb{N}$ and $a\in[p-1]$. For the remainder of this article, we denote the elements of $\Lin(P_{ap^k})$ as follows. 
	\begin{itemize}
		\item Let $\mathsf{u}=(\mathsf{u}_1,\dotsc,\mathsf{u}_k)\in\{0,1\}^k$. Then $\phi(\mathsf{u})$ denotes any element $\mathcal{X}(\psi_1,\dotsc,\psi_k)$ of $\Lin(P_{p^k})$ such that $\mathsf{u}_j=1$ if and only if $\psi_j=\triv_{P_p}$.
		
		\item Let $\underline{\mathbf{u}}=(\mathsf{u}^1,\dotsc,\mathsf{u}^a)$ with $\mathsf{u}^i=(\mathsf{u}^i_1,\dotsc,\mathsf{u}^i_k)\in\{0,1\}^k$ for each $i\in[a]$. Then $\phi(\underline{\mathbf{u}})$ denotes any element $\mathcal{L}(\psi_{1,1},\dotsc,\psi_{a,k})$ of $\Lin(P_{ap^k})$ such that $\mathsf{u}^i_j=1$ if and only if $\psi_{i,j}=\triv_{P_p}$.
	\end{itemize}
	For short, we will write $\phi(\mathsf{u})\in\Lin(P_{p^k})$ or $\phi(\underline{\mathbf{u}})\in\Lin(P_{ap^k})$, meaning that $\mathsf{u}\in\{0,1\}^k$ or $\underline{\mathbf{u}}=(\mathsf{u}^1,\dotsc,\mathsf{u}^a)$ with $\mathsf{u}^i\in\{0,1\}^k$ for each $i\in[a]$ respectively.\hfill$\lozenge$
\end{notation}

With this, Theorem~\ref{thm: apk} is therefore equivalent to the following:
\begin{theorem}\label{thm: apk restated}
	Let $k\in\mathbb{N}$ and $a\in[p-1]$. Let $\phi(\underline{\mathbf{s}})$, $\phi(\underline{\mathbf{t}})\in\Lin(P_{ap^k})$. Suppose $\phi(\underline{\mathbf{s}})\up^{\fS_{ap^k}}=\phi(\underline{\mathbf{t}})\up^{\fS_{ap^k}}$. Then there exists a permutation $\sigma\in\fS_a$ such that $\mathsf{s}^i=\mathsf{t}^{\sigma(i)}$ for all $i$.
\end{theorem}

\bigskip

\section{Induced character values}\label{sec:char values} 

Throughout this section, fix $p$ any prime and $k\in\mathbb{N}$.
To complete the proof of Theorem~\ref{thm:A}, it remains to prove Theorem~\ref{thm: apk restated}. Our approach will be to calculate the values $\phi(\underline{\mathbf{s}})\up^{\fS_{ap^k}}(g)$ for $g$ belonging to various different conjugacy classes of $\fS_n$, or in other words, for $g$ of various different cycle types. Our goal will be to express these values in terms of the sequences $\mathsf{s}^i$ of $\underline{\mathbf{s}}$, so that we can use the assumption $\phi(\underline{\mathbf{s}})\up^{\fS_{ap^k}}=\phi(\underline{\mathbf{t}})\up^{\fS_{ap^k}}$ to deduce information relating $\underline{\mathbf{s}}$ and $\underline{\mathbf{t}}$.
Recall that 
\begin{equation}\label{eqn:induced}
\phi\up^{\fS_{ap^k}}_{P_{ap^k}}(g) = \frac{|{\bf C}_{\fS_{ap^k}}(g)|}{|P_{ap^k}|}\cdot \sum_{x\in\ccl_{\fS_{ap^k}}(g)\cap P_{ap^k}}\phi(x).
\end{equation}
Here and throughout, ${\bf C}_G(g)$ denotes the centraliser and $\ccl_G(g)$ the conjugacy class in $G$ of an element $g\in G$. Since it is easy to calculate $|{\bf C}_{\fS_{ap^k}}(g)|$ and $|P_{ap^k}|$, our focus is on understanding the summation term in \eqref{eqn:induced}, and therefore the structure of $\ccl_{\fS_{ap^k}}(g)\cap P_{ap^k}$. By writing $P_{ap^k}$ as the direct product of $a$ copies of $P_{p^k}$, we can reduce this task to understanding $\ccl_{\fS_{p^k}}(h)\cap P_{p^k}$ for various $h\in\fS_{p^k}$. Since $P_{p^k}=P_{p^{k-1}}\wr P_p$, we can tackle this by induction on $k$. However, the wreath product structure makes this a rather complicated task in general.
	
Fortunately, in order to deduce Theorem~\ref{thm: apk restated}, it will in fact suffice to consider $\phi(\underline{\mathbf{s}})\up^{\fS_{ap^k}}(g)$ for only those $g$ which, excluding fixed points, are products of cycles whose lengths are distinct powers of $p$. In this case the calculations simplify considerably, as explained in the following remark.

\begin{remark}\label{rem: newremark}
	The reason for the aforementioned simplification is exemplified by the case $n=p^k$. Given an element $g\in \fS_{p^k}$ which, excluding fixed points, is a product of disjoint cycles of \emph{distinct} $p$-power length, either $g$ has at least one fixed point in $[{p^k}]$ or $g$ is a single $p^k$-cycle. In the first case, the presence of a fixed point implies that $\ccl_{\fS_{p^k}}(g)\cap P_{p^k}$ is contained in the natural subgroup $P_{p^{k-1}}\times\cdots\times P_{p^{k-1}}$ of $P_{p^k}$, allowing us to avoid working with the wreath product structure. The second case will be dealt with using Lemma~\ref{lem: pk cycles}, which characterises $p^k$-cycles in $P_{p^k}$.
		\hfill$\lozenge$
\end{remark}

The structure of this section is as follows. We begin with the case $n=p^k$ and $g$ a single $p^l$-cycle, for some $1\le l \le k$. In this case, the value $\phi\up^{\fS_{p^k}}_{P_{p^k}}(g)$ (for $\phi\in\Lin(P_{p^k})$) is given by Lemma~\ref{lem: b=1}. This lemma will be proved by induction on $k$, with the key ideas as described in the paragraph above. In Proposition~\ref{prop: b>1}, we extend these methods to any $g\in \fS_{p^k}$ a product of disjoint cycles of {distinct} $p$-power length. Theorem~\ref{thm:general-a} then generalises Proposition~\ref{prop: b>1} from $n=p^k$ to $n=ap^k$ with $a\in[p-1]$ by using the fact that $P_{ap^k}=P_{p^k}\times \dots\times P_{p^k}$. This is conceptually straightforward, but there are several combinatorial simplifications needed to reduce to a usable formula.

\medskip

We are now ready to begin the technical details. We start with some definitions which will help ease the notation in the rest of this section.

\begin{definition}\label{def:Gamma}
	Let $b\in\mathbb{N}$ and suppose that $l_1,l_2,\dotsc,l_b$ are natural numbers in $[k]$. Suppose $g\in\fS_{p^k}$ has disjoint cycles of length $p^{l_1}, p^{l_2},\cdots, p^{l_b}, 1, \dotsc,1$ (we also say $g$ \emph{has cycle type} $p^{l_1}p^{l_2}\cdots p^{l_b}$). Let $\mathsf{u}\in\{0,1\}^k$. Define
	\[ \Gamma_{l_1,l_2,\dotsc,l_b;k}(\mathsf{u}) := \sum_{x\in\ccl_{\fS_{p^k}}(g)\cap P_{p^k}} \phi(\mathsf{u})(x). \]
	More generally, we define $\Gamma_{l_1,l_2,\dotsc,l_b;k}(\mathsf{u})$ for any natural numbers $l_1,\dotsc,l_b$ by setting the value to be 0 whenever $p^{l_1}+p^{l_2}+\cdots+p^{l_b}>p^k$. 
	In particular, \eqref{eqn:induced} can be rewritten as
	\begin{equation}\label{eqn:induced2}
	\phi(\mathsf{u})\up^{\fS_{p^k}}_{P_{p^k}}(g) = \tfrac{|{\bf C}_{\fS_{p^k}}(g)|}{|P_{p^k}|}\cdot\Gamma_{l_1,\dotsc,l_b;k}(\mathsf{u}).
	\end{equation}
\end{definition}

\begin{definition}\label{def:CXY}
	\begin{itemize}
		\item[(i)] Let $l\in[k]$ and $\mathsf{u}=(\mathsf{u}_1,\dotsc,\mathsf{u}_k)\in\{0,1\}^k$. We define
		\[ C_l(\mathsf{u}) := p^{\tfrac{p^l-1}{p-1}-2l}\cdot\prod_{m=1}^l(p\mathsf{u}_m-1). \]
		
		\item[(ii)] For a multiset $T$ of natural numbers, we let $X(T)$ (respectively $Y(T)$) denote the set of elements of $P_{p^k}$ (respectively $P_{p^{k-1}}$) of cycle type $\prod_{t\in T}p^t$.
		If $T=\{t_1,\dotsc,t_z\}$, we sometimes also use the notation $X(t_1,\dotsc,t_z)$ or $Y(t_1,\dotsc,t_z)$.
	\end{itemize}
\end{definition}

It is useful to observe that $C_l(\mathsf{u})$ does not depend on $\mathsf{u}_{l+1},\dotsc,\mathsf{u}_k$. In other words, $C_l(\mathsf{u})=C_l\big((\mathsf{u}_1,\dotsc,\mathsf{u}_l)\big)$. The following lemma gives the value of $\phi(\mathsf{u})\up^{\fS_{p^k}}_{P_{p^k}}$ on a single $p^l$-cycle, using \eqref{eqn:induced2}.

\begin{lemma}\label{lem: b=1}
	Let $l\in[k]$ and let $\mathsf{u}=(\mathsf{u}_1,\dotsc,\mathsf{u}_k)\in\{0,1\}^k$. Then 
	\[ \Gamma_{l;k}(\mathsf{u}) = p^k\cdot C_l(\mathsf{u}). \]
\end{lemma}

\begin{proof}
	Observe that if $l=k=1$, then $\Gamma_{1;1}(\mathsf{u})=p\mathsf{u}_1-1$. In other words,
	\[ \Gamma_{1;1}(\mathsf{u}) = \sum_{x\in P_p\setminus\{1\}} \phi(\mathsf{u})(x) = \begin{cases} p-1 & \mathrm{if}\ \phi(\mathsf{u})=\triv_{P_p}\ \ (\text{i.e.~if}\ \mathsf{u}=(1)),\\ -1 & \text{otherwise \ \ (i.e.~if $\mathsf{u}=(0)$)}.\end{cases} \]
	Now let $k\ge 2$ and let $\mathsf{u}^-=(\mathsf{u}_1,\dotsc,\mathsf{u}_{k-1})$. First suppose $l=k$. 
	Let $x\in X(k)$, that is, let $x$ be a $p^k$-cycle in $P_{p^k}$. By Lemma~\ref{lem: pk cycles}, $x=(f_1,\dotsc,f_p;\sigma)$ where $\sigma\in P_p\setminus\{1\}$, $f_i\in P_{p^{k-1}}$ for all $i$ and $f_{\sigma^{p-1}(1)}\cdots f_{\sigma(1)}\cdot f_1\in Y(k-1)$.
	By Lemma~\ref{lem:4.3.9},
	
	\begin{align*}
	\Gamma_{k;k}(\mathsf{u}) &= 
	\sum_{x\in X(k)} \mathcal{X}(\mathsf{u}^-;\mathsf{u}_k) \big((f_1,\dotsc,f_p;\sigma)\big) = \sum_{x\in X(k)} \phi(\mathsf{u}^-)(f_{\sigma^{p-1}(1)}\cdots f_{\sigma(1)}\cdot f_1)\cdot\phi(\mathsf{u}_k)(\sigma)\\
	&= |P_{p^{k-1}}|^{p-1}\cdot \sum_{y\in Y(k-1)} \phi(\mathsf{u}^-)(y) \cdot \sum_{\sigma\in P_p\setminus\{1\}} \phi(\mathsf{u}_k)(\sigma) = p^{p^{k-1}-1}\cdot \Gamma_{k-1;k-1}(\mathsf{u}^-)\cdot(p\mathsf{u}_k-1).
	\end{align*}
	The third equality holds since for any fixed $y\in Y(k-1)$, we may choose the $p-1$ elements $f_1,f_{\sigma(1)},\dotsc,f_{\sigma^{p-2}(1)}$ in $P_{p^{k-1}}$ freely, after which $f_{\sigma^{p-1}(1)}\cdots f_{\sigma(1)}\cdot f_1=y$ uniquely determines $f_{\sigma^{p-1}(1)}$. Inductively, we have that
	\[ \Gamma_{k;k}(\mathsf{u}) = p^{(p^{k-1}+p^{k-2}+\cdots+1)-k}\cdot\prod_{m=1}^k (p\mathsf{u}_m-1) = p^{\tfrac{p^k-1}{p-1}-k}\cdot\prod_{m=1}^k (p\mathsf{u}_m-1). \]
	Next, let $l\in[k-1]$. Let $x\in X(l)$. Then $x$ must have a fixed point as $l<k$. Thus $x=(f_1,\dotsc,f_p;1)$ where $f_i\in Y(l)$ for a unique $i\in[p]$ and $f_j=1$ for all $j\ne i$. Hence by Lemma~\ref{lem:4.3.9},
	\begin{align*}
	\Gamma_{l;k}(\mathsf{u}) &= \sum_{x\in X(l)} \phi(\mathsf{u}^-)(f_1)\cdots\phi(\mathsf{u}^-)(f_p)\cdot\phi(\mathsf{u}_k)(1)\\
	&= \binom{p}{1}\sum_{y\in Y(l)} \phi(\mathsf{u}^-)(y) = p\cdot \Gamma_{l;k-1}(\mathsf{u}^-) = p^{k-l}\cdot\Gamma_{l;l}\big((\mathsf{u}_1,\dotsc,\mathsf{u}_l)\big).
	\end{align*}
	The assertion then follows from the $l=k$ case.
\end{proof}

\begin{corollary}\label{cor:non-zero}
	Let $a\in[p-1]$ and $\phi\in\Lin(P_{ap^k})$. Then $\phi\up^{\fS_{ap^k}}_{P_{ap^k}}(g)\ne 0$ where $g$ is a product of $a$ disjoint $p^k$-cycles.
\end{corollary}

\begin{proof}
	From \eqref{eqn:induced}, it suffices to show that
	\[ \sum_{x\in\ccl_{\fS_{ap^k}}(g)\cap P_{ap^k}} \phi(x) \ne 0. \]
	Since $P_{ap^k}=P_{p^k}\times\cdots\times P_{p^k}$ ($a$ times), $x\in P_{ap^k}$ implies $x=x_1\cdots x_a$ with each $x_i$ in a distinct direct factor $P_{p^k}$. Letting $\phi=\phi(\mathsf{u}^1)\times\cdots\times\phi(\mathsf{u}^a)$ where $\phi(\mathsf{u}^i)\in\Lin(P_{p^k})$ for each $i$, we therefore have that
	\[ \sum_{x\in\ccl_{\fS_{ap^k}}(g)\cap P_{ap^k}} \phi(x) = \prod_{i=1}^a \left(\sum_{x_i\in X(k)}\phi(\mathsf{u}^i)(x_i)\right) = \prod_{i=1}^a \Gamma_{k;k}(\mathsf{u}^i) = \prod_{i=1}^a \big(p^k\cdot C_k(\mathsf{u}^i)\big) \]
	by Lemma~\ref{lem: b=1}, which is non-zero (Definition~\ref{def:CXY}).
\end{proof}

\begin{remark}\label{rmk:a>1}
	Lemma~\ref{lem: b=1} is already enough to prove Theorem~\ref{thm: apk restated} when $a=1$. Indeed, let $\phi(\mathsf{s})$, $\phi(\mathsf{t})\in\Lin(P_{p^k})$ for some $\mathsf{s},\mathsf{t}\in\{0,1\}^k$ and suppose $\phi(\mathsf{s})\up^{\fS_{p^k}}=\phi(\mathsf{t})\up^{\fS_{p^k}}$. Then $\phi(\mathsf{s})\up^{\fS_{p^k}}(g)=\phi(\mathsf{t})\up^{\fS_{p^k}}(g)$ for each $g\in\fS_{p^k}$, in particular $g$ of cycle type $p^l$ for each $l\in[k]$. By Lemma~\ref{lem: b=1}, this implies
	\[ \prod_{m=1}^l (p\mathsf{s}_m-1) = \prod_{m=1}^l (p\mathsf{t}_m-1)\]
	for all $l\in[k]$. Therefore $\mathsf{s}_m=\mathsf{t}_m$ for all $m\in[k]$, and thus $\mathsf{s}=\mathsf{t}$.
	
	However, Lemma~\ref{lem: b=1} is not enough when $a>1$. For example, let $a=2$, $k=3$ and consider $\underline{\mathbf{s}}=\big((1,0,0),(0,1,1)\big)$ and $\underline{\mathbf{t}}=\big((1,0,1),(0,1,0)\big)$. The induced characters $\phi(\underline{\mathbf{s}})\up^{\fS_{2p^3}}$ and $\phi(\underline{\mathbf{t}})\up^{\fS_{2p^3}}$ agree on $p$, $p^2$ and $p^3$-cycles, but are not equal. This motivates considering more complicated cycle types, as we shall now do.
	\hfill$\lozenge$
\end{remark}

%

The next proposition gives the value of the induced character $\phi(\mathsf{u})\up^{\fS_{p^k}}_{P_{p^k}}$ on a product of $p$-power cycles of distinct length. Before we state the proposition, we give some key definitions.

\begin{notation}\label{not:l}
	\begin{itemize}\itemspace
		\item[(i)] For a set $A$, let $\partition{A}$ denote the set of partitions of $A$, e.g.
		\[ \partition{\{1,2,3\}} = \Big\{ \big\{ \{1,2,3\} \big\}, \big\{ \{1,2\}, \{3\} \big\}, \big\{ \{1,3\}, \{2\} \big\}, \big\{ \{2,3\}, \{1\} \big\}, \big\{ \{1\}, \{2\}, \{3\} \big\} \Big\}. \]
		
		\item[(ii)] For the remainder of Section~\ref{sec:char values}, suppose $l_1,l_2,\dotsc,l_b$ are fixed integers for some $b\in\mathbb{N}$ such that $1\le l_1<l_2<\cdots<l_b\le k$. We further let $p_j:=p^{l_j}$.\hfill$\lozenge$
	\end{itemize}
\end{notation}

\begin{definition}\label{def:R}
	Let $\tau=\{t_1<t_2<\cdots<t_z\}\subseteq[b]$ be a set of distinct natural numbers. Let $i\in\mathbb{N}$. Define
	\[ R^\tau_i:= \begin{cases}
	(p^i-p_{t_z})(p^i-p_{t_z}-p_{t_{z-1}})\cdots(p^i-p_{t_z}-\cdots-p_{t_2}) & \text{if }z\ge 2,\\
	1 & \text{if }z=1.
	\end{cases} \]
\end{definition}

\noindent For example, if $b>1$ then $R^{[b]}_k = (p^k-p^{l_b})(p^k-p^{l_b}-p^{l_{b-1}})\cdots(p^k-p^{l_b}-p^{l_{b-1}}-\cdots-p^{l_2})$.

\begin{proposition}\label{prop: b>1}
	Let $b\in[p-1]$ and $\mathsf{u}\in\{0,1\}^k$. Suppose $l_1,l_2,\dotsc,l_b$ are integers such that $1\le l_1<l_2<\cdots<l_b\le k$. Then
	\begin{equation}\label{eqn: prop b>1}
	\Gamma_{l_1,l_2,\dotsc,l_b;k}(\mathsf{u}) = p^k\cdot C_{l_1}(\mathsf{u})\cdots C_{l_b}(\mathsf{u})\cdot R^{[b]}_k.
	\end{equation}
\end{proposition}

In Appendix~\ref{sec:appendix}, we provide detailed examples illustrating the proof of Proposition~\ref{prop: b>1} for $b=2$ and $b=3$. In Example~\ref{ex:gamma b=2} we prove the proposition directly for $b=2$, and in Example~\ref{ex:gamma b=3} we demonstrate the inductive step of the argument by deducing the proposition for $b=3$ from the case of $b=2$. 
Equations (\ref{prop: b>1}(a)--(g)) below are labelled as such to illustrate their correspondence to equations (\ref{ex:gamma b=3}(a)--(g)) in Appendix~\ref{sec:appendix}.
It may help the reader to first familiarise with the main proof ideas by reading Examples~\ref{ex:gamma b=2} and~\ref{ex:gamma b=3} before reading the general proof, which we now give. 

\begin{proof}[Proof of Proposition~\ref{prop: b>1}]
	We proceed by induction on $b$, with Lemma~\ref{lem: b=1} providing the base case $b=1$. Now assume $b\ge 2$, and that Proposition~\ref{prop: b>1} holds for all $b'\in[p-1]$ such that $b'<b$. Notice that both sides of \eqref{eqn: prop b>1} equal 0 if $l_b=k$, so from now on we may assume $l_b<k$. Let $\mathsf{u}^-=(\mathsf{u}_1,\dotsc,\mathsf{u}_{k-1})$.
	From Definitions~\ref{def:Gamma} and~\ref{def:CXY},
	\begin{equation}\label{eqn:b>2 start}
	\tag{\ref{prop: b>1}(a)}
	\Gamma_{l_1,\dotsc,l_b;k}(\mathsf{u}) = \sum_{x\in X(l_1,\dotsc,l_b)}\phi(\mathsf{u})(x).
	\end{equation}
	Consider some $x\in X(l_1,\dotsc,l_b)$. Since $l_1<l_2<\cdots <l_b<k$, we deduce that $x$ has a fixed point. Hence $x=(f_1,\dots,f_p;1)$ with $f_1,\dotsc,f_p\in P_{p^{k-1}}$ and so 
	\[ \phi(\mathsf{u})(x) = \phi(\mathsf{u}^-)(f_1)\cdot \phi(\mathsf{u}^-)(f_2)\cdots\phi(\mathsf{u}^-)(f_p) \]
	by Lemma~\ref{lem:4.3.9}. Moreover, the product of the cycle types of $f_i$ over all $i$ equals $p^{l_1}\cdots p^{l_b}$ since $x\in X(l_1,\dotsc,l_b)$.
	By considering how the cycle lengths $p^{l_w}$ ($w\in[b]$) are grouped together to give the cycle types of $f_1,\dotsc,f_p$
	, we can rewrite the sum over $x\in X(l_1,\dotsc,l_b)$ in \eqref{eqn:b>2 start} as a sum over partitions 
	$\nu$ of the set $[b]=\{1,2,\dotsc,b\}$
	to obtain 
	\begin{equation}\label{eqn:b>2 (a)}
	\tag{\ref{prop: b>1}(b)}
	\begin{split}
	\Gamma_{l_1,\dotsc,l_b;k}(\mathsf{u}) &= \sum_{\nu\in\partition{[b]}} \frac{p!}{(p-|\nu|)!}\cdot\prod_{\omega\in\nu} \Gamma_{{l_\omega};k-1}(\mathsf{u}^-)\\
	&= \sum_{\substack{\nu\in\partition{[b]}\\ \nu\ne\{[b]\}}} \frac{p!}{(p-|\nu|)!}\cdot\prod_{\omega\in\nu} \Gamma_{{l_\omega};k-1}(\mathsf{u}^-) + p\cdot\Gamma_{l_1,\dotsc,l_b;k-1}(\mathsf{u}^-).
	\end{split}
	\end{equation}
	Here $\Gamma_{{l_\omega};k-1}(\mathsf{u}^-)$ denotes $\Gamma_{l_{w_1},\dotsc,l_{w_t};k-1}(\mathsf{u}^-)$ when $\omega=\{w_1,\dotsc,w_t\}\subseteq[b]$. (So the last term of \eqref{eqn:b>2 (a)} which corresponds to $\nu=\{[b]\}$ can also be written as $p\cdot\Gamma_{l_{[b]};k-1}(\mathsf{u}^-)$, for instance.)
	By iterating \eqref{eqn:b>2 (a)} we obtain the following:
	\begin{equation}\label{eqn:b>2 (b)}
	\tag{\ref{prop: b>1}(c)}
	\Gamma_{l_1,\dotsc,l_b;k}(\mathsf{u}) = \sum_{i=l_b}^{k-1} p^{k-i} \sum_{\substack{\nu\in\partition{[b]}\\ \nu\ne\{[b]\}}} \frac{(p-1)!}{(p-|\nu|)!}\cdot\prod_{\omega\in\nu} \Gamma_{{l_\omega};i}\big((\mathsf{u}_1,\dotsc,\mathsf{u}_i)\big),
	\end{equation}
	recalling that $\Gamma_{l_1,l_2,\dotsc,l_b;i}=0$ if $l_b=i$.
	Since $\nu\ne\{[b]\}$, every $\omega$ appearing in \eqref{eqn:b>2 (b)} satisfies $|\omega|<b$. Therefore the inductive hypothesis gives
	\[ \Gamma_{{l_\omega};i}\big((\mathsf{u}_1,\dotsc,\mathsf{u}_i)\big) = p^i\cdot C_{l_{w_1}}(\mathsf{u})\cdots C_{l_{w_t}}(\mathsf{u})\cdot R^\omega_i. \]
	Substituting this into \eqref{eqn:b>2 (b)}, we obtain
	\begin{align}
	\Gamma_{l_1,\dotsc,l_b;k}(\mathsf{u}) &= \sum_{i=l_b}^{k-1} p^{k-i} \sum_{\substack{\nu\in\partition{[b]}\\ \nu\ne\{[b]\}}} \frac{(p-1)!}{(p-|\nu|)!}\cdot\prod_{\omega\in\nu} \left[ p^i \cdot \prod_{w\in\omega} C_{l_w}(\mathsf{u}) \cdot R^\omega_i \right]\nonumber\\
	&= p^k\cdot \prod_{i=1}^b C_{l_i}(\mathsf{u})\cdot \sum_{i=l_b}^{k-1}\ \sum_{\substack{\nu\in\partition{[b]}\\ \nu\ne\{[b]\}}} \frac{(p-1)!}{(p-|\nu|)!}\cdot p^{i(|\nu|-1)}\cdot\prod_{\omega\in\nu} R^\omega_i.	\tag{\ref{prop: b>1}(d)}
	\end{align}
	Thus, it remains to prove that $\Sigma_b = R^{[b]}_k$, 
	where we define $\Sigma_b$ as follows:
	\begin{equation}
	\tag{\ref{prop: b>1}(e)}
	\Sigma_b:= \sum_{i=l_b}^{k-1}\ \sum_{\substack{\nu\in\partition{[b]}\\ \nu\ne\{[b]\}}} \frac{(p-1)!}{(p-|\nu|)!}\cdot p^{i(|\nu|-1)}\cdot\prod_{\omega\in\nu} R^\omega_i.
	\end{equation}
	Let
	\[ Q_i := \sum_{\substack{ \gamma\in\partition{\{2,3,\dotsc,b\}}\\ \gamma\ne\{\{2,3,\dotsc,b\}\} }} \frac{(p-1)!}{(p-|\gamma|)!}\cdot p^{i(|\gamma|-1)}\cdot\prod_{\omega\in\gamma} R^\omega_i. \]
	Since Proposition~\ref{prop: b>1} holds for $b-1$ (for $l_2<l_3<\cdots<l_b$), we have 
	for all $i\in\mathbb{N}$ such that $i>l_b$ 
	the following:
	\begin{equation}\label{eqn:prop-f}
	\tag{\ref{prop: b>1}(f)}
	\sum_{j=l_b}^{i-1} Q_j = (p^i-p^{l_b})(p^i-p^{l_b}-p^{l_{b-1}})\cdots(p^i-p^{l_b}-\cdots-p^{l_3}) = R^{\{2,3,\dotsc,b\}}_i.
	\end{equation}
	Now we return to $\Sigma_b$, and sum together the terms whose associated $\nu\in\partition{[b]}$ correspond to the same $\gamma\in\partition{\{2,3,\dotsc,b\}}$ upon removing the number 1. Moreover, for a fixed $\gamma$, we first group together those terms with $\nu$ of the form (a) $\nu=\big(\gamma\setminus\{\tau\}\big) \cup \{ \tau\cup\{1\} \}$ for some $\tau\in\gamma$, and separately (b) $\nu=\gamma\cup\{\{1\}\}$. Notice $|\nu|=|\gamma|$ or $|\nu|=|\gamma|+1$ respectively. 
	In particular, in case (a), letting $\tau=\{w_1,\cdots,w_t\}\subseteq\{2,3,\dotsc,b\}$ then
	\[ \prod_{\omega\in\nu} R^\omega_i = \prod_{\omega\in\gamma} R^\omega_i \cdot(p^i-p^{l_{w_t}}-p^{l_{w_{t-1}}}-\cdots-p^{l_{w_1}}), \]
	while in case (b) we have $\prod_{\omega\in\nu} R^\omega_i = \prod_{\omega\in\gamma} R^\omega_i$. Finally, observe that since $\nu\in\partition{[b]}$ with $\nu\ne\{[b]\}$, then the only $\nu$ whose corresponding $\gamma$ equals $\{2,3,\dotsc,b\}$ is $\nu=\{ \{1\}, \{2,3,\dotsc,b\} \}$.
	Letting $B=\{2,3,\dotsc,b\}$, we thus obtain 
	\begin{align}
	\Sigma_b &= \sum_{i=l_b}^{k-1} \Bigg[ \sum_{\substack{ \gamma\in\partition{B}\\ \gamma\ne\{B\} }}
	\Bigg( \underbrace{\tfrac{(p-1)!}{(p-|\gamma|)!}\cdot p^{i(|\gamma|-1)}\cdot  \prod_{\omega\in\gamma} R^\omega_i \cdot \sum_{\tau\in\gamma} \Big(p^i-\sum_{w\in\tau} p^{l_w}\Big)}_{ \nu=(\gamma\setminus\{\tau\}) \cup \{ \tau\cup\{1\} \}\text{ for some }\tau\in\gamma }
	+ \underbrace{\tfrac{(p-1)!}{(p-|\gamma|-1)!}\cdot p^{i|\gamma|}\cdot \prod_{\omega\in\gamma} R^\omega_i}_{\nu=\gamma\cup\{\{1\}\}} \Bigg)\nonumber\\
	&\qquad\qquad + \underbrace{\tfrac{(p-1)!}{(p-2)!}\cdot p^i\cdot R^{\{2,3,\dotsc,b\}}_i}_{\gamma=\{B\}} 
	\Bigg]\nonumber\\
	&= \sum_{i=l_b}^{k-1} \Bigg[ \sum_{\substack{ \gamma\in\partition{B}\\ \gamma\ne\{B\} }} 
	\tfrac{(p-1)!\cdot p^{i(|\gamma|-1)} }{(p-|\gamma|)!}\prod_{\omega\in\gamma} R^\omega_i \Bigg( \sum_{\tau\in\gamma} \Big(p^i-\sum_{w\in\tau} p^{l_w}\Big) + (p-|\gamma|)p^i \Bigg) + (p-1)p^i R^{\{2,\dotsc,b\}}_i \Bigg]\nonumber\\
	&= \sum_{i=l_b}^{k-1} \Bigg[ Q_i 
	\cdot (p^{i+1}-p^{l_b}-\cdots-p^{l_2}) 
	+ (p-1)p^i\cdot R^{\{2,3,\dotsc,b\}}_i \Bigg]\nonumber\\
	&= (-p^{l_b}-\cdots-p^{l_2})\sum_{i=l_b}^{k-1} Q_i + \sum_{i=l_b}^{k-1} \left[ p^{i+1}Q_i + (p-1)p^i\sum_{j=l_b}^{i-1} Q_j \right] \quad (\text{by \eqref{eqn:prop-f}})\nonumber\\
	&= (-p^{l_b}-\cdots-p^{l_2})\sum_{i=l_b}^{k-1} Q_i + p^k \cdot \sum_{h=l_b}^{k-1} Q_h\nonumber\\
	&=(p^k-p^{l_b}-\cdots-p^{l_2}) \cdot (p^k-p^{l_b})(p^k-p^{l_b}-p^{l_{b-1}})\cdots(p^k-p^{l_b}-\cdots-p^{l_3}), \tag{\ref{prop: b>1}(g)}
	\end{align}
	which concludes the proof.
\end{proof}

The final proposition of this section gives the value of the induced character $\phi(\underline{\mathbf{u}})\up^{\fS_{ap^k}}_{P_{ap^k}}$ on a product of $p$-power cycles of distinct length where $a\in[p-1]$. We first fix some useful notation.

\begin{definition}\label{def:W}
	Suppose $a\in[p-1]$ is fixed. For a multiset $T=\{t_1,\dotsc,t_z\}$ of natural numbers, we let $W(T)$ (or $W(t_1,\dotsc,t_z)$) denote the set of elements of $P_{ap^k}$ of cycle type $\prod_{t\in T}p^t$.
\end{definition}

\begin{notation}\label{not:prop}
	For the remainder of Section~\ref{sec:char values}, we fix the following notation. Recall $l_1,l_2,\dotsc,l_b$ are fixed integers such that $1\le l_1<l_2<\cdots<l_b\le k$. Suppose further that $b\in[a]$ and $a\in[p-1]$. Fix $\phi(\underline{\mathbf{u}})\in\Lin(P_{ap^k})$ with $\underline{\mathbf{u}}=(\mathsf{u}^1,\dotsc,\mathsf{u}^a)$ and $\mathsf{u}^i\in\{0,1\}^k$ for each $i$.
	
	\smallskip
	
	For the function $C_l(\mathsf{u})$ introduced in Definition~\ref{def:CXY} (i), we simplify the notation by making the following abbreviations: 
	\[ C_t(i):=C_{l_t}(\mathsf{u}^i)\qquad \text{and}\qquad  C_rC_s\cdots C_t(i):=C_r(i)\cdot C_s(i)\cdots C_t(i). \]
	Furthermore, we let
	\[ (\Sum C_rC_s\cdots C_t) := \sum_{i=1}^a C_rC_s\cdots C_t(i) \]
	when the summation index runs through all of $\{1,2,\dotsc,a\}$. 
	\hfill$\lozenge$
\end{notation}

\begin{definition}\label{def:Mtau}
	Recall that $p_j:=p^{l_j}$. For $\tau=\{t_1<t_2<\cdots<t_z\}\subseteq[b]$ a set of distinct natural numbers, we write
	\[ M(\tau):=\begin{cases}
	p_{t_z}\cdot(p_{t_z}+p_{t_{z-1}})\cdots(p_{t_z}+\cdots+p_{t_2}) & \text{if }z\ge 2,\\
	1 & \text{if }z=1.
	\end{cases} \]
\end{definition}


\begin{theorem}\label{thm:general-a}
	Using Notation~\ref{not:prop}, we let $g\in\fS_{ap^k}$ be of cycle type $(p^{l_1}, p^{l_2},\ldots,p^{l_b})$ and we set ${\bf C}={\bf C}_{\fS_{ap^k}}(g)$. Then we have that
	\[ \frac{|P_{ap^k}|}{|{\bf C}|}\cdot\phi(\underline{\mathbf{u}})\up^{\fS_{ap^k}}(g)= \sum_{\nu\in\partition{[b]}} (-1)^{b-|\nu|}\cdot p^{k|\nu|}\cdot \prod_{\tau\in\nu}\Bigg[ M(\tau)\cdot \sum_{i=1}^a \Big( \prod_{t\in\tau} C_{t}(i)\Big) \Bigg]. \]
\end{theorem}

\bigskip

\noindent Before starting with the proof of Theorem \ref{thm:general-a}, we remark that since \[ \frac{|P_{ap^k}|}{|{\bf C}|}\cdot\phi(\underline{\mathbf{u}})\up^{\fS_{ap^k}}(g)=\sum_{x\in W(l_1,\dotsc,l_b)} \phi(\underline{\mathbf{u}})(x),\]
our main aim in the next few pages will be to show that 
\[ \sum_{x\in W(l_1,\dotsc,l_b)} \phi(\underline{\mathbf{u}})(x) = \sum_{\nu\in\partition{[b]}} (-1)^{b-|\nu|}\cdot p^{k|\nu|}\cdot \prod_{\tau\in\nu}\Bigg[ M(\tau)\cdot \sum_{i=1}^a \Big( \prod_{t\in\tau} C_{t}(i)\Big) \Bigg]. \]

\noindent For example, when $b=3$, the equation above (and more generally Theorem~\ref{thm:general-a}) asserts that

\begin{small}
	\begin{multline*}
	\sum_{x\in W(l_1,l_2,l_3)} \phi(\underline{\mathbf{u}})(x) = p^{3k}(\Sum C_1)(\Sum C_2)(\Sum C_3) - p^{2k}\cdot p_2(\Sum C_1C_2)(\Sum C_3)
	-p^{2k}\cdot p_3(\Sum C_1C_3)(\Sum C_2)\\ -p^{2k}\cdot p_3(\Sum C_2C_3)(\Sum C_1) + p^k\cdot p_3(p_3+p_2)(\Sum C_1C_2C_3).
	\end{multline*}
\end{small}

\noindent We remark that the case of $a=1$ follows immediately from Lemma~\ref{lem: b=1}, so from now on we may assume $a\ge 2$. 
We now begin the proof of Theorem~\ref{thm:general-a}, which forms the remainder of this section.

\subsection{Proof of Theorem \ref{thm:general-a}}
Let $x\in W(l_1,\dotsc,l_b)$. Since $P_{ap^k}\cong P_{p^k}\times\cdots\times P_{p^k}$ ($a$ times), we can write $x=x_1\cdots x_a$ with each $x_i$ in a distinct direct factor $P_{p^k}$. Then $\phi(\underline{\mathbf{u}})(x) = \phi(\mathsf{u}^1)(x_1)\cdots\phi(\mathsf{u}^a)(x_a)$. Since $x\in W(l_1,\dotsc,l_b)$, the cycle types of $x_1,\dotsc,x_a$ 
naturally give a partition of the set $[b]=\{1,2,\dotsc,b\}$.
We regroup the sum over $x\in W(l_1,\dotsc,l_b)$ according to the corresponding partition of $[b]$ and substitute in Proposition~\ref{prop: b>1}, in order to express $\sum_x \phi(\underline{\mathbf{u}})(x)$ in terms of the function $C_t(i)$. 

\smallskip

For example, when $b=3$ we have (recall the five partitions of the set $[3]$)
\begin{small}
	\begin{align}\label{eqn:b=3 example}
	\sum_{x\in W(l_1,l_2,l_3)} \phi(\underline{\mathbf{u}})(x) &= \sum_{\substack{h,i,j\in[a]\\\text{distinct}}} \Gamma_{l_1;k}(\mathsf{u}^h)\Gamma_{l_2;k}(\mathsf{u}^i)\Gamma_{l_3;k}(\mathsf{u}^j) + \sum_{i\ne j\in[a]}\Big[ \Gamma_{l_1,l_2;k}(\mathsf{u}^i)\Gamma_{l_3;k}(\mathsf{u}^j) \nonumber\\
	&\quad + \Gamma_{l_1,l_3;k}(\mathsf{u}^i)\Gamma_{l_2;k}(\mathsf{u}^j) +
	\Gamma_{l_2,l_3;k}(\mathsf{u}^i)\Gamma_{l_1;k}(\mathsf{u}^j) \Big] + \sum_{i=1}^a \Gamma_{l_1,l_2,l_3;k}(\mathsf{u}^i) \nonumber\\
	= p^{3k}\sum_{\substack{h,i,j\in[a]\\\text{distinct}}} &C_1(h)C_2(i)C_3(j) + p^{2k}\sum_{i\ne j\in[a]}\Big[ (p^k-p_2) C_1C_2(i)C_3(j) + (p^k-p_3) C_1C_3(i)C_2(j) \nonumber\\
	&\quad + (p^k-p_3) C_2C_3(i)C_1(j) \Big] + p^k(p^k-p_3)(p^k-p_3-p_2)\sum_{i=1}^a C_1C_2C_3(i).
	\end{align}
\end{small}

\noindent For general $b$, we obtain
\begin{equation}\label{eqn:KC*}
\sum_{x\in W(l_1,\dotsc,l_b)} \phi(\underline{\mathbf{u}})(x) = \sum_{\nu\in\partition{[b]}} \kappa_\nu\cdot\mathcal{C}^*_\nu
\end{equation}
where we define the terms $\kappa_\nu$ and $\mathcal{C}^*_\nu$ as follows.

\begin{definition}\label{def:KC}
	Let $\nu\in\partition{[b]}$.
	\begin{itemize}\itemspace
		\item[(i)] Define $\kappa_\nu:= p^{k|\nu|}\cdot\prod_{\tau\in\nu} R^\tau_k$.
		
		\item[(ii)] Let $\nu=\{\nu_1,\dotsc,\nu_s\}$. Define $\mathcal{C}_\nu$ to be the following sum with unrestricted indices (i.e.~free to run independently over $1,2,\dotsc,a$)
		\[ \mathcal{C}_\nu:= \sum_{i_1,\dotsc,i_s=1}^a (\textstyle\prod_{t\in \nu_1}C_t)(i_1)\cdot (\textstyle\prod_{t\in \nu_2}C_t)(i_2)\cdots (\textstyle\prod_{t\in \nu_s}C_t)(i_s). \]
		In other words, $\mathcal{C}_\nu = \prod_{\omega\in\nu} (\Sum \textstyle\prod_{t\in\omega}C_t)$. 
		For example, 
		\[ \mathcal{C}_{\{ \{1,2,4\},\{3,5\},\{6\}\}} = (\Sum C_1C_2C_4)(\Sum C_3C_5)(\Sum C_6).\]
		
		\item[(iii)] Let $\nu=\{\nu_1,\dotsc,\nu_s\}$. Define $\mathcal{C}^*_\nu$ to be the following sum with restricted indices (i.e.~satisfying a distinctness condition)
		\[ \mathcal{C}^*_\nu:= \sum_{\substack{i_1,\dotsc,i_s\in[a]\\\text{distinct}}} (\textstyle\prod_{t\in \nu_1}C_t)(i_1)\cdot (\textstyle\prod_{t\in \nu_2}C_t)(i_2)\cdots (\textstyle\prod_{t\in \nu_s}C_t)(i_s). \]
		
		\item[(iv)] We define a partial order $\le$ on $\partition{[b]}$ as follows. For $\nu,\lambda\in\partition{[b]}$, we say $\nu\le\lambda$ if $\lambda$ may be obtained from $\nu$ by taking unions of some of the parts of $\nu$. In other words, $\nu$ is a finer partition of $[b]$ than $\lambda$.
		For example,
		\[ \nu=\big\{ \{1,2\}, \{3\}, \{4\}, \{5\}, \{6\} \big\} \le \big\{ \{1,2,3\}, \{4,5,6\} \big\} =\lambda. \]
		For $\nu\le\lambda$ and $\tau\in\lambda$, we define $m^\lambda_\nu(\tau)$ to be the number of parts of $\nu$ which make up $\tau$. In the above example, if $\tau=\{1,2,3\}$ then $m^\lambda_\nu(\tau)=2$, while if $\tau=\{4,5,6\}$ then $m^\lambda_\nu(\tau)=3$.
	\end{itemize}
\end{definition}

\noindent Equation~\eqref{eqn:KC*} provides an expression for the induced character value $\phi(\underline{\mathbf{u}})\up_{P_{ap^k}}^{\fS_{ap^k}}$ on elements of cycle type $p^{l_1}\cdots p^{l_b}$ in terms of the sums $\mathcal{C}^*_\nu$ in which the indices are restricted by imposing the distinctness condition. This condition makes equation~\eqref{eqn:KC*} difficult to apply directly for our target of distinguishing induced characters based on their indexing parameter ${\underline{\mathbf{u}}}$. It turns out to be much easier to relate the unrestricted sums $\mathcal{C}_\nu$ to ${\underline{\mathbf{u}}}$, and our first objective will be to express $\mathcal{C}^*_\nu$ in terms of various $\mathcal{C}_\lambda$.
To that end, it is easy to see that
\[ \mathcal{C}_\lambda = \sum_{\nu\ge\lambda} \mathcal{C}^*_\nu. \]
What we would like is an inverse: to express $\mathcal{C}^*_\nu$ as a sum of $\mathcal{C}_\lambda$ for certain $\lambda$, and then to substitute this into \eqref{eqn:KC*}. 
For instance, we may replace the restricted indices $i\ne j\in[a]$ by unrestricted indices in the $\nu=\{\{1,2\},\{3\}\}$ term of \eqref{eqn:b=3 example} by observing that
\begin{small}
	\[ \sum_{i\ne j\in[a]} C_1C_2(i)C_3(j) = \Bigg(\sum_{i=1}^a C_1C_2(i)\Bigg)\cdot\Bigg(\sum_{j=1}^a C_3(j)\Bigg) - \sum_{i=1}^a C_1C_2C_3(i) = (\Sum C_1C_2)(\Sum C_3) - (\Sum C_1C_2C_3).\]
\end{small}

\noindent Consider the final two terms above as corresponding to the two members $\omega$ of $\partition{\nu}$, namely $\omega=\big\{ \{ \alpha \}, \{ \beta \} \big\}$ and $\omega=\big\{ \{ \alpha,\beta \} \big\}$ respectively, where $\nu=\{\alpha,\beta\}$, $\alpha=\{1,2\}$ and $\beta=\{3\}$.
We can perform a similar replacement in general for any number of restricted indices; this is the content of the following lemma.

\begin{lemma}\label{lem:replace}
	Let $\nu\in\partition{[b]}$. Then
	\begin{equation}\label{eqn:C*}
	\mathcal{C}^*_\nu = \sum_{\nu\le\lambda\in\partition{[b]}} (-1)^{|\nu|-|\lambda|}\prod_{\tau\in\lambda} (m^\lambda_\nu(\tau)-1)!\cdot \mathcal{C}_\lambda.
	\end{equation}
\end{lemma}

\begin{proof}
	Let $\mathcal{D}_\nu$ be the expression on the right hand side of \eqref{eqn:C*}. We wish to prove that $\mathcal{C}^*_\nu=\mathcal{D}_\nu$.
	
	Suppose $|\nu|=n$. By identifying the members of $\nu$ with $\{1\}, \dotsc, \{n\}$, we may without loss of generality assume $\nu=\{\{1\},\dotsc,\{n\}\}$, the minimal element of $\partition{[n]}$ under $\le$. Since $\mathcal{C}_\lambda=\sum_{\alpha\ge\lambda} \mathcal{C}^*_\alpha$, we have
	\begin{align}\label{eqn:D=C*}
	\mathcal{D}_\nu = \sum_{\alpha\in\partition{[n]}} \Bigg(\sum_{\lambda:\lambda\le\alpha}(-1)^{n-|\lambda|}\cdot\prod_{\tau\in\lambda}(m^\lambda_\nu(\tau)-1)!\Bigg)\mathcal{C}^*_\alpha.
	\end{align}
	If $\alpha=\nu$, then we obtain the term $1\cdot\mathcal{C}^*_\nu$. 
	
	Now fix $\alpha\ne\nu$. 
	The coefficient of $\mathcal{C}^*_\alpha$ in \eqref{eqn:D=C*} may be rewritten as
	\[ \sum_{\lambda:\lambda\le\alpha}(-1)^{n-|\lambda|}\cdot\prod_{\tau\in\lambda}(m^\lambda_\nu(\tau)-1)! = (-1)^{n-2|\lambda|}\cdot\prod_{\tau\in\alpha}\Bigg( \sum_{\rho\in\partition{\tau}} (-1)^{|\rho|}\prod_{t\in\rho}(|t|-1)!\Bigg). \]
	For any fixed $\tau\in\alpha$, if $|\tau|=m$ then 
	\[ \sum_{\rho\in\partition{\tau}} (-1)^{|\rho|}\prod_{t\in\rho}(|t|-1)! 
	= B_m(-0!,-1!,\dotsc,-(m-1)!). \]
	Here $B_m(x_1,\dotsc,x_m)$ denotes the $m^{th}$ complete Bell polynomial: the coefficient of each monomial $\prod_i x_i^{a_i}$ is the number of partitions of $[m]$ such that there are exactly $a_i$ members of size $i$, e.g.~$B_4(x_1,x_2,x_3,x_4)=x_1^4+6x_1^2x_2+4x_1x_3+3x_2^2+x_4$. Moreover, by the determinantal form $B_m(x_1,\dotsc,x_m)=\det \mathsf{B}$ where $\mathsf{B}$ is the $m\times m$ matrix with
	\[ \mathsf{B}_{i,i+j}=\tfrac{x_{j+1}}{j!}\ \forall j\ge 0,\quad \mathsf{B}_{i,i-1}=-i+1\quad\text{and}\quad \mathsf{B}_{i,j}=0\text{ otherwise}, \]
	it is clear that if $m\ge 2$ then $B_m(x_1,\dotsc,x_m)=0$ when $x_i=-(i-1)!$ for all $i$. Since $\alpha\ne\nu$, there exists some $\tau\in\alpha$ with $|\tau|\ge 2$. The assertion of the lemma then follows from \eqref{eqn:D=C*}.
\end{proof}

\begin{corollary}\label{cor:theta}
	We have that
	\[ \sum_{x\in W(l_1,\dotsc,l_b)} \phi(\underline{\mathbf{u}})(x) = \sum_{\lambda\in\partition{[b]}} \theta_\lambda\cdot\mathcal{C}_\lambda,\]
	where
	\[ \theta_\lambda := \sum_{\nu\le\lambda} p^{k|\nu|} \cdot\prod_{\tau\in\nu} R^\tau_k\cdot(-1)^{|\nu|-|\lambda|}\prod_{z\in\lambda}\big( m^\lambda_\nu(\tau)-1 \big)!.\]
\end{corollary}

\begin{proof}
	This follows from Definition~\ref{def:KC} and substituting Lemma~\ref{lem:replace} into \eqref{eqn:KC*}.
\end{proof}

Comparing with the assertion of Theorem~\ref{thm:general-a}, it therefore remains to prove that
\begin{equation}\label{eqn:rtp}
\theta_\lambda = (-1)^{b-|\lambda|}\cdot p^{k|\lambda|}\cdot \prod_{\tau\in\lambda} M(\tau).
\end{equation}
To show that \eqref{eqn:rtp} holds, we view $\theta_\lambda$ as a polynomial in the term $p^k$ and evaluate the coefficients as follows. Fix $\lambda\in\partition{[b]}$. For $u\ge|\lambda|$, we define the quantities $T_u$ via:
\[ \theta_\lambda=: \sum_{u\ge |\lambda|} p^{uk}\cdot T_u. \]
For example, let $b=3$ and $\lambda=\{\{1,2,3\}\}$. Then the set of those $\nu\le\lambda$ is all of $\partition{[3]}$, and so
\begin{small}
	\begin{equation*}
	\begin{array}{lll}
	\theta_\lambda&  =\sum_{\nu\le\lambda} p^{k|\nu|} \cdot\prod_{\tau\in\nu} R^\tau_k\cdot(-1)^{|\nu|-|\lambda|}\prod_{z\in\lambda}\big( m^\lambda_\nu(\tau)-1 \big)! &\\[4pt]
	&= p^k(p^k-p_3)(p^k-p_3-p_2) & \nu=\{\{1,2,3\}\}\\[4pt]
	&\ \ -p^{2k}(p^k-p_2) & \nu=\{\{1,2\},\{3\}\}\\[4pt]
	&\ \ -p^{2k}(p^k-p_3) & \nu=\{\{1,3\},\{2\}\}\\[4pt]
	&\ \ -p^{2k}(p^k-p_3) & \nu=\{\{2,3\},\{1\}\}\\[4pt]
	&\ \ + 2p^{3k} & \nu=\{\{1\},\{2\},\{3\}\}\\[4pt]
	\implies &\multicolumn{2}{l}{ T_1 = p_3(p_3+p_2),\quad T_2=(-2p_3-p_2)+p_2+p_3+p_3,\quad T_3=1-1-1-1+2. }
	\end{array}
	\end{equation*}
\end{small}

\noindent This example and \eqref{eqn:rtp} motivate the following lemma.

\begin{lemma}\label{lem:T}
	Fix $\lambda\in\partition{[b]}$. Then 
	$T_{|\lambda|} = (-1)^{b-|\lambda|}\prod_{\tau\in\lambda}M(\tau)$ and $T_u=0$ for all $u>|\lambda|$.
\end{lemma}

In order to prove Lemma~\ref{lem:T}, we introduce some definitions.

\begin{definition}\label{def:Mcuts}
	Let $n\in\mathbb{N}$, $\alpha\in\partition{[n]}$ and $i\in\mathbb{N}_0$.
	\begin{itemize}\itemspace
		\item[(i)] Define $\mathcal{M}_i(\alpha)$ to be the coefficient of $p^{ik}$ in $\prod_{\tau\in\alpha} R^\tau_k$ (see Definition~\ref{def:R}), viewed as a polynomial in $p^k$. In particular, $\mathcal{M}_0(\alpha) = \prod_{\tau\in\alpha} (-1)^{|\tau|-1}M(\tau)=(-1)^{n-|\alpha|}\prod_{\tau\in\alpha}M(\tau)$ (see Definition~\ref{def:Mtau} for $M(\tau)$). 
		
		For example, when $\alpha=\big\{ \{1,2\}, \{3,4,5\} \big\}\in \partition{[5]}$, $\mathcal{M}_0(\alpha)=-p_2p_5(p_5+p_4)$ which is the coefficient of $p^{0\cdot k}$ in $(p^k-p_2)(p^k-p_5)(p^k-p_5-p_4)$.
		
		\item[(ii)] Define $\cut(\alpha,i) = \{\beta\in\partition{[n]} : \beta\le\alpha\text{ and }|\beta|=|\alpha|+i \}$. In other words, we may think of $\beta$ as being obtained from $\alpha$ by cutting $i$ times, e.g.
		\[ \alpha=\big\{ \{1,2\}, \{3,4,5\} \big\},\qquad \beta = \big\{ \{1\}, \{2\}, \{4\}, \{3,5\} \big\} \in \cut(\alpha,2). \]
		In particular, $\cut(\{[n]\},i) = \{\beta\in\partition{[n]} : |\beta|=1+i \}$.
	\end{itemize}
\end{definition}

\begin{remark}\label{rem:partial}
	We have that 
	\begin{equation}\label{eq:product}
	\mathcal M_0(\{[n]\})=(-1)^{n-1}M([n])=(-1)^{n-1}\prod_{a=2}^n(p_n+p_{n-1}+\dots+p_a) 
	\end{equation}
	Similarly, the coefficient $\mathcal{M}_i(\{[n]\})$ is equal to $(-1)^{n-1-i}$ times the sum of all products of $n-1-i$ of the terms in the product on the right hand side of \eqref{eq:product}. Since the term $p_n$ appears in every bracket, we may express this sum as a partial derivative with respect to $p_n$ in the following way. 
	Treat $M([n])$ as a polynomial in indeterminates $p_j$, and treat $\mathcal{M}_i(\{[n]\})$ similarly. Then 
	\[ \mathcal{M}_i(\{[n]\}) = -\frac{1}{i}\frac{\partial}{\partial p_n}\mathcal{M}_{i-1}(\{[n]\}), \]
	since each choice of $n-1-i$ terms from the product on the right hand side of \eqref{eq:product} can be obtained in $i$ different ways by first choosing $n-i$ terms and then discarding one of them.\hfill$\lozenge$
\end{remark}

\begin{lemma}\label{lem:M}
	Let $n\in\mathbb{N}$ and $i\in\mathbb{N}_0$. Then
	\[ \mathcal{M}_i(\{[n]\}) = \sum_{\mu\in \cut(\{[n]\},i)} \mathcal{M}_0(\mu). \]
\end{lemma}

\begin{proof}
	Since the maximal power of $p^k$ in $R^{[n]}_k$ is $p^{(n-1)k}$, the statement is true when $i\ge n$ as $\cut(\{[n]\},i)=\emptyset$. Moreover, the cases $n\le 2$ and $i=0$ are trivial, so we may now assume $n\ge 3$ and $i\in[n-1]$. 
	
	We proceed by induction on $i$: first suppose $i=1$.
	Let $a\in\{2,3,\dotsc,n\}$ and write $J_a$ for the expression obtained by dropping the factor $(p_n+\dotsc+p_a)$ from the product $M([n])$. 
	That is,
	\begin{align*}
	J_a:&=\frac{M([n])}{p_n+p_{n-1}+\dotsc+p_a}\\
	&=p_n\cdots(p_n+p_{n-1}+\cdots+p_{a+1})(p_n+p_{n-1}+\cdots+p_{a-1})\cdots(p_n+p_{n-1}+\cdots+p_2).
	\end{align*}
	Recalling Remark~\ref{rem:partial}, we have that $\mathcal{M}_1(\{[n]\})=\sum_{a=2}^n (-1)^{n-2}J_a$. We claim that 
	\[ J_a=\Xi_a,\]
	where we define
	\[ \Xi_a := \sum_{Y\subseteq[a-2]} M\Big(\Big\{Y\cup\{a-1\}\Big\}\Big)\cdot M\Big(\Big\{[n]\setminus(Y\cup\{a-1\})\Big\}\Big) \]
	with $[0]:=\emptyset$. The assertion of the present lemma for $i=1$ then follows by summing $J_a=\Xi_a$ over $a\in\{2,3,\dotsc,n\}$, since $\Xi_a$ is precisely $(-1)^{n-2}\mathcal{M}_0(\mu)$ summed over those $\mu\in\cut(\{[n]\},1)$ such that, writing $\mu=\{\mu_1,\mu_2\}$ where $n\in\mu_2$, we have $\max(\mu_1)=a-1$.
	
	\smallskip
	
	We prove $J_a=\Xi_a$ by induction on $a$. Consider first the base case $a=2$. In this case, observe $J_2=p_n(p_n+p_{n-1})\cdots(p_n+p_{n-1}+\cdots+p_3)$ by definition, while
	\[ \Xi_2 = M(\{1\})\cdot M\big(\big\{[n]\setminus\{1\}\big\}\big)=1\cdot p_n(p_n+p_{n-1})\cdots(p_n+p_{n-1}+\cdots+p_3).\]
	For $a>2$, we partition the sum over $Y\subseteq[a-2]$ in $\Xi_a$ according to whether the element $a-2$ belongs to $Y$, so that $\Xi_a=\Xi_{a,1}+\Xi_{a,2}$ where 
	\[ \Xi_{a,1}=\sum_{Y\subseteq[a-3]} M\Big(\Big\{Y\cup\{a-1\}\Big\}\Big)\cdot M\Big(\Big\{[n]\setminus(Y\cup\{a-1\})\Big\}\Big)\]
	and
	\[ \Xi_{a,2}=\sum_{Y\subseteq[a-3]} M\Big(\Big\{Y\cup\{a-1,a-2\}\Big\}\Big)\cdot M\Big(\Big\{[n]\setminus(Y\cup\{a-1,a-2\})\Big\}\Big).\]
	Observe now that replacing all instances of $p_{a-1}$ by $p_{a-2}$ and vice versa in the expression for $\Xi_{a,1}$ yields $\Xi_{a-1}$. But the inductive hypothesis implies that $\Xi_{a-1}=J_{a-1}$, and we note that in the product
	\[ J_{a-1}= p_n(p_n+p_{n-1})\cdots(p_n+p_{n-1}+\cdots+p_{a})(p_n+p_{n-1}+\cdots+p_{a-2})\cdots(p_n+p_{n-1}+\cdots+p_2)\]
	the terms $p_{a-1}$ and $p_{a-2}$ appear in the same brackets. Therefore swapping $p_{a-1}$ and $p_{a-2}$ in the expression for $J_{a-1}$ leaves it unchanged, and so $\Xi_{a,1}=J_{a-1}$.
	
	Next, observe that $M(\{Y\cup\{a-1,a-2\}\})$ is equal to $p_{a-1}$ times the expression obtained by replacing all occurrences of $p_{a-2}$ with $(p_{a-1}+p_{a-2})$ in $M(\{Y\cup\{a-2\}\})$ -- this is easily seen from Definition~\ref{def:Mtau}. Moreover, by our induction hypothesis we have that
	\begin{equation}\label{eqn:sub}
	\sum_{Y\subseteq[a-3]} M\Big(\Big\{Y\cup\{a-2\}\Big\}\Big)\cdot M\Big(\Big\{[n]\setminus(Y\cup\{a-1,a-2\})\Big\}\Big)=\frac{M\Big([n]\setminus\{a-1\}\Big)}{p_n+p_{n-1}+\dotsc+p_a}.
	\end{equation}
	Replacing all occurrences of $p_{a-2}$ with $(p_{a-1}+p_{a-2})$ on the right hand side of \eqref{eqn:sub} yields 
	\[ \frac{M([n])}{(p_n+p_{n-1}+\dotsc+p_a)(p_n+p_{n-1}+\dotsc+p_a+p_{a-1})} =\frac{J_{a-1}}{p_n+p_{n-1}+\dotsc+p_a}\]
	and so 
	\[ \Xi_{a,2} = p_{a-1}\frac{J_{a-1}}{p_n+p_{n-1}+\dotsc+p_a} = J_a-J_{a-1}. \]
	So $\Xi_a=\Xi_{a,1}+\Xi_{a,2}=J_a$ as required. This proves $J_a=\Xi_a$ for all $a\in\{2,3,\dotsc,n\}$, and hence the case of $i=1$ is concluded.
	
	\smallskip
		
	Now suppose $i>1$.
	As in Remark~\ref{rem:partial}, we may express $\mathcal{M}_i(\{[n]\})$ in terms of $\tfrac{\partial}{\partial p_n}\mathcal{M}_{i-1}(\{[n]\})$, and using the induction hypothesis we obtain
	\begin{align*}
	\mathcal{M}_i(\{[n]\}) &= -\frac{1}{i}\frac{\partial}{\partial p_n}\mathcal{M}_{i-1}(\{[n]\}) = -\frac{1}{i} \sum_{\mu\in\cut(\{[n]\},i-1)} \frac{\partial}{\partial p_n}\mathcal{M}_0(\mu)\\
	&= \frac{1}{i}\sum_{\mu\in\cut(\{[n]\},i-1)} \mathcal{M}_1(\{\mu_1\})\cdot \mathcal{M}_0(\{\mu_2\})\cdots\mathcal{M}_0(\{\mu_i\})
	\end{align*}
	\noindent where $\mu:=\{\mu_1,\dotsc,\mu_i\}$ and $n\in\mu_1$.
	By using the inductive hypothesis on $\mathcal{M}_1(\{\mu_1\})$, we obtain
	\begin{align*}
	\mathcal{M}_i(\{[n]\}) &= \frac{1}{i} \sum_{\mu\in\cut(\{[n]\},i-1)}\  \sum_{\substack{\nu=\{\nu_1,\nu_2\}\\\in\cut(\{\mu_1\},1)}} \mathcal{M}_0(\{\nu_1\})\cdot \mathcal{M}_0(\{\nu_2\})\cdot \prod_{j=2}^i \mathcal{M}_0(\{\mu_j\}) \\
	&= \frac{1}{i} \sum_{\gamma\in\cut(\{[n]\},i)} N_\gamma\cdot \mathcal{M}_0(\gamma)
	\end{align*}
	
	\noindent where $N_\gamma$ is the number of ways of obtaining $\gamma\in\cut(\{[n]\},i)$ by first taking some $\mu\in\cut(\{[n]\},i-1)$ then replacing the part $\mu_1$ of $\mu$ containing the number $n$ by $\nu_1$, $\nu_2$, where $\{\nu_1,\nu_2\}\in\cut(\{\mu_1\},1)$. But for all $\gamma$, we have $N_\gamma=i$ since $|\gamma|=i+1$ and $\mu$ is uniquely determined once one of $\{\tau\in\gamma:n\notin\tau\}$ is chosen to be $\nu_2$ (assuming without loss of generality that $n\in\nu_1$). This concludes the proof.
\end{proof}

\begin{corollary}\label{cor:M}
	Let $n\in\mathbb{N}$ and $i\in\mathbb{N}_0$. Let $\alpha\in\partition{[n]}$. Then
	\[ \mathcal{M}_i(\alpha) = \sum_{\beta\in \cut(\alpha,i)} \mathcal{M}_0(\beta). \]
\end{corollary}

\begin{proof}
	We proceed by induction on $|\alpha|$, with Lemma~\ref{lem:M} giving the base case $|\alpha|=1$. In particular, we may now assume $|\alpha|\ge 2$. Fix some $\tau\in\alpha$. Then
	\begin{align*}
	\sum_{\beta\in \cut(\alpha,i)} \mathcal{M}_0(\beta) &= \sum_{j\ge 0}\  \sum_{\substack{\beta\in\cut(\alpha,i)\text{ where $\tau$ is}\\\text{cut exactly j times}}} \mathcal{M}_0(\beta) =\sum_{j\ge 0}\ \sum_{\substack{\gamma\in\cut(\{\tau\},j),\\\delta\in\cut(\alpha\setminus\{\tau\},i-j)}} \mathcal{M}_0(\gamma)\cdot\mathcal{M}_0(\delta)\\ 
	&= \sum_{j\ge 0} \mathcal{M}_j(\{\tau\})\cdot\mathcal{M}_{i-j}(\alpha\setminus\{\tau\}) = \mathcal{M}_i(\alpha)
	\end{align*}
	where the third equality follows from Lemma~\ref{lem:M} and the inductive hypothesis.
\end{proof}

\begin{proof}[Proof of Lemma~\ref{lem:T}]
	The value of $T_{|\lambda|}$ follows immediately from Corollary~\ref{cor:theta}, so from now on assume $u>|\lambda|$. By Corollary~\ref{cor:theta} and Definition~\ref{def:Mcuts}, the coefficient in $\theta_\lambda$ of $p^{uk}$ is
	\begin{align}\label{eqn:Tu}
	T_u &= \sum_{\substack{\nu\le\lambda\\|\nu|\le u}} \mathcal{M}_{u-|\nu|}(\nu)\cdot (-1)^{|\nu|-|\lambda|} \prod_{\tau\in\lambda}\big(m^\lambda_\nu(\tau)-1\big)!\nonumber\\
	&= \sum_{\substack{\nu\le\lambda\\|\nu|\le u}}\ \sum_{\beta\in\cut(\nu, u-|\nu|)} \mathcal{M}_0(\beta)\cdot (-1)^{|\nu|-|\lambda|} \prod_{\tau\in\lambda}\big(m^\lambda_\nu(\tau)-1\big)!\nonumber\\
	&= \sum_{\substack{\beta\le\lambda\\|\beta|=u}} \mathcal{M}_0(\beta)\cdot  \sum_{\nu: \beta\le\nu\le\lambda} (-1)^{|\nu|-|\lambda|} \prod_{\tau\in\lambda}\big(m^\lambda_\nu(\tau)-1\big)!\nonumber\\
	&= \sum_{\substack{\beta\le\lambda\\|\beta|=u}} \mathcal{M}_0(\beta)\cdot \prod_{\tau\in\lambda}\  \sum_{\sigma\in\partition{[m^\lambda_\beta(\tau)]}} (-1)^{|\sigma|-1}(|\sigma|-1)!
	\end{align}
	where the second equality follows from Corollary~\ref{cor:M}. 
	From the following well-known identity involving Stirling numbers of the second kind $S(n,j)$ and the polylogarithm $\operatorname{Li}$,
	\[ \sum_{j=1}^n S(n,j)\cdot(j-1)!w^j = (-1)^n\operatorname{Li}_{1-n}(1+\tfrac{1}{w}), \]
	we obtain 
	\[ \sum_{\sigma\in\partition{[n]}} (-1)^{|\sigma|}(|\sigma|-1)! = 0 \]
	for all $n\in\mathbb{N}$. Substituting this into \eqref{eqn:Tu} with $n=m^\lambda_\beta(\tau)$, we obtain $T_u=0$ as desired.
\end{proof}

Since \eqref{eqn:rtp} follows immediately from Lemma~\ref{lem:T}, the proof of Theorem~\ref{thm:general-a} is finally concluded.

\bigskip

\section{Proof of Theorem~\ref{thm: apk restated}}\label{sec:final}

Theorem~\ref{thm:general-a} gives us a formula with which we can evaluate $\phi(\underline{\mathbf{s}})\up^{\fS_{ap^k}}$ on elements of certain cycle types. The strategy for the proof of Theorem~\ref{thm: apk restated} is as follows.
	
We first use Theorem~\ref{thm:general-a} to translate the equality of characters \[ \phi(\underline{\mathbf{s}})\up^{\fS_{ap^k}}=\phi(\underline{\mathbf{t}})\up^{\fS_{ap^k}}\]
into a collection of equations relating the parameters $\underline{\mathbf{s}}$ and $\underline{\mathbf{t}}$ (see Corollary~\ref{cor:4.18}). These equations, however, are in terms of the functions $C$ described in Definition~\ref{def:CXY}, and it remains to relate these functions to the entries of the sequences in $\underline{\mathbf{s}}$ and $\underline{\mathbf{t}}$ themselves. This will be achieved in Proposition~\ref{prop: claim 8} following two technical lemmas.

Proposition~\ref{prop: claim 8} tells us that if $\phi(\underline{\mathbf{s}})\up^{\fS_{ap^k}}=\phi(\underline{\mathbf{t}})\up^{\fS_{ap^k}}$ then we have equality between multisets of values obtained from certain partial sums of entries of $\underline{\mathbf{s}}$ and $\underline{\mathbf{t}}$. It still remains to show these equalities of partial sums provide enough information to conclude that Theorem~\ref{thm: apk restated} holds. This will follow from a combinatorial argument encapsulated within Lemma~\ref{lem:keylemma} and Theorem~\ref{thm: part2}.

\begin{corollary}\label{cor:4.18}
	Let $a\in[p-1]$. Let $\phi(\underline{\mathbf{s}})$, $\phi(\underline{\mathbf{t}})\in\Lin(P_{ap^k})$. Suppose that $\phi(\underline{\mathbf{s}})\up^{\fS_{ap^k}}=\phi(\underline{\mathbf{t}})\up^{\fS_{ap^k}}$. Let $b\in[a]$ and suppose $l_1,l_2,\dotsc,l_b$ are integers such that $1\le l_1<\cdots<l_b\le k$. Then
	\[ \sum_{j=1}^a C_{l_1}(\mathsf{s}^j)\cdot C_{l_2}(\mathsf{s}^j)\cdots C_{l_b}(\mathsf{s}^j) = \sum_{j=1}^a C_{l_1}(\mathsf{t}^j)\cdot C_{l_2}(\mathsf{t}^j)\cdots C_{l_b}(\mathsf{t}^j).\]
\end{corollary}

\begin{proof}
	We proceed by induction on $b$. First suppose $b=1$ and let $g$ be an element of $\fS_{ap^k}$ of cycle type $p^{l_1}$. By Theorem~\ref{thm:general-a}, 
	\[ \phi(\underline{\mathbf{s}})\up^{\fS_{ap^k}}(g) = \tfrac{|{\bf C}_{\fS_{ap^k}}(g)|}{|P_{ap^k}|} \cdot \sum_{x\in W(l_1)} \phi(\underline{\mathbf{s}})(x) = \tfrac{|{\bf C}_{\fS_{ap^k}}(g)|}{|P_{ap^k}|} \cdot p^k\cdot \sum_{i=1}^a C_{l_1}(\mathsf{s}^i), \]
	and similarly for $\underline{\mathbf{t}}$. Since $\phi(\underline{\mathbf{s}})\up^{\fS_{ap^k}}=\phi(\underline{\mathbf{t}})\up^{\fS_{ap^k}}$, the assertion follows.
	
	Now suppose $b\ge 2$. Let $g\in\fS_{ap^k}$ of cycle type $p^{l_1}\cdots p^{l_b}$. By Theorem~\ref{thm:general-a},
	\begin{equation}\label{eqn:equatingC}
	\phi(\underline{\mathbf{s}})\up^{\fS_{ap^k}}(g) = \tfrac{|{\bf C}_{\fS_{ap^k}}(g)|}{|P_{ap^k}|} \cdot \sum_{\nu\in\partition{[b]}} (-1)^{b-|\nu|}p^{k|\nu|}\prod_{\tau\in\nu}\Big[M(\tau)\cdot\sum_{i=1}^a\Big(\prod_{z\in\tau} C_{l_z}(\mathsf{s}^i)\Big)\Big].
	\end{equation}
	By the inductive hypothesis, $\sum_{i=1}^a (\prod_{z\in\tau}C_{l_z}(\mathsf{s}^i)) = \sum_{i=1}^a (\prod_{z\in\tau}C_{l_z}(\mathsf{t}^i))$ for every $\tau$ satisfying $|\tau|<b$. The only $\tau$ such that $|\tau|=b$ comes from $\nu=\{[b]\}$ (giving $\tau=[b]$). 
	Using $\phi(\underline{\mathbf{s}})\up^{\fS_{ap^k}}=\phi(\underline{\mathbf{t}})\up^{\fS_{ap^k}}$ and substituting these observations into \eqref{eqn:equatingC}, we obtain
	\[ (-1)^{b-1}p^k \cdot M([b])\cdot\sum_{i=1}^a C_{l_1}(\mathsf{s}^i)\cdots C_{l_b}(\mathsf{s}^i) = (-1)^{b-1}p^k \cdot M([b])\cdot\sum_{i=1}^a C_{l_1}(\mathsf{t}^i)\cdots C_{l_b}(\mathsf{t}^i). \]
	Since $(-1)^{b-1}p^k \cdot M([b])\ne 0$, the assertion follows.
\end{proof}

From Corollary~\ref{cor:4.18} we would like to deduce Proposition~\ref{prop: claim 8} which states, informally, that certain partial sums of the sequences $\mathsf{s}^1,\dotsc,\mathsf{s}^a$ and $\mathsf{t}^1,\dotsc,\mathsf{t}^a$ are equal. Recall from Definition~\ref{def:CXY} that the part of $C_l(\mathsf{u})$ depending on $\mathsf{u}$ is 
\[ \prod_{m=1}^{l}(p\mathsf{u}_m-1) = (-1)^l\cdot(1-p)^{\sum_{m=1}^l \mathsf{u}_m}.\]
Corollary~\ref{cor:4.18} thus implies equality between certain sums of powers of $1-p$. The following technical lemma will be used (with $q=p-1$) to help translate these equalities into more directly applicable information.

For notational convenience, we denote multisets by asterisks. 

\begin{lemma}\label{lem: signed powers}
	Let $q\in\mathbb{N}_{\ge 2}$. Let $a\in[q]$ and let $\sigma_j, \tau_j\in\mathbb{N}_0$ for $j\in[a]$. If
	\[ \sum_{j=1}^a (-q)^{\sigma_j} = \sum_{j=1}^a (-q)^{\tau_j}, \]
	then either 
	\begin{itemize}
	\item[(i)] $\{\sigma_1,\dotsc,\sigma_a\}^*=\{\tau_1,\dotsc,\tau_a\}^*$ is an equality of multisets; or 
	\item[(ii)] $a=q$ and the multisets $\{\sigma_1,\dotsc,\sigma_a\}^*$ and $\{\tau_1,\dotsc,\tau_a\}^*$ are  $\{w,w-1,\dotsc,w-1\}^*$ and $\{w-2,\dotsc,w-2\}^*$ for some $w\in\mathbb{N}_{\ge 2}$.
	\end{itemize}
\end{lemma}

\begin{proof}
	We proceed by induction on $a$. The assertion is clear if $a=1$, so now assume $2\le a\le q$, and suppose $\{\sigma_1,\dotsc,\sigma_a\}^*\ne\{\tau_1,\dotsc,\tau_a\}^*$. If $\sigma_i=\tau_j$ for some $i,j\in[a]$, then by the inductive hypothesis for $a-1\ne q$, we have $\{\sigma_1,\dotsc, \sigma_{i-1}, \sigma_{i+1},\dotsc,\sigma_a\}^*=\{\tau_1,\dotsc,\tau_{j-1},\tau_{j+1},\dotsc,\tau_a\}^*$. But then $\{\sigma_1,\dotsc,\sigma_a\}^*=\{\tau_1,\dotsc,\tau_a\}^*$, a contradiction. Thus $\sigma_i\ne \tau_j$ for all $i,j\in[a]$. Without loss of generality suppose $\sigma_1=\max\{\sigma_i,\tau_j\}_{i,j\in[a]}$, so in particular $\tau_j<\sigma_1$ for all $j$. By multiplying by $-q$ if necessary (equivalently, by adding 1 to all $\sigma_i$ and $\tau_j$), we may further assume that $\sigma_1$ is even. Then
	\[ q^{\sigma_1-1}\le q^{\sigma_1}-(a-1)q^{\sigma_1-1}\le \sum_{j=1}^a q^{\sigma_j}(-1)^{\sigma_j} = \sum_{j=1}^a q^{\tau_j}(-1)^{\tau_j} \le aq^{\sigma_1-2} \le q^{\sigma_1-1}. \]
	Hence all inequalities in the above must hold with equalities, implying that $a=q$, $\sigma_j=\sigma_1-1$ for all $j\ne 1$, and $\tau_j=\sigma_1-2$ for all $j\in[a]$. This is exactly case (ii).
\end{proof}

Unfortunately, the existence of case (ii) of Lemma~\ref{lem: signed powers} creates a modest technical hurdle. In order to rule out this case in our application of Lemma~\ref{lem: signed powers} to the equations from Corollary~\ref{cor:4.18}, we will provide one additional equation which will imply, informally, that the sequences $\mathsf{s}^1,\dotsc,\mathsf{s}^a$ and $\mathsf{t}^1,\dotsc,\mathsf{t}^a$ have equal partial sums. This is the content of Lemma~\ref{lem: claim 7} below. In order to provide this additional equation, we use one more explicit computation of induced character values, in a case not covered by Theorem~\ref{thm:general-a}: the case where $g$ is the product of two disjoint $p^l$-cycles. 
	
This computation follows very similar lines to the calculations in Section~\ref{sec:char values}. This case is no more difficult because when $n=p^k$, $p\ge 3$ and $g$ is a product of two $p^l$-cycles, then $g$ has a fixed point in $[n]$ and the inductive argument will avoid working with the wreath product structure as before (see Remark~\ref{rem: newremark}). The assumption in the following Lemma that $p$ is an odd prime is therefore crucial, since when $p=2$ the situation is complicated by the fact that a product of two disjoint cycles of length $2^{k-1}$ in $\fS_{2^k}$ has no fixed point.

\begin{lemma}\label{lem: claim 7}
	Let $p$ be an odd prime, $a\in\{2,3,\dotsc,p-1\}$, $k\in\mathbb{N}$ and $l\in[k]$. Let $\phi(\underline{\mathbf{s}})$, $\phi(\underline{\mathbf{t}})\in\Lin(P_{ap^k})$ and suppose $\phi(\underline{\mathbf{s}})\up^{\fS_{ap^k}}=\phi(\underline{\mathbf{t}})\up^{\fS_{ap^k}}$. Then
	\[ \left\{ \sum_{m=1}^l \mathsf{s}^1_m, \dotsc, \sum_{m=1}^l \mathsf{s}^a_m \right\}^* =\ \left\{ \sum_{m=1}^l \mathsf{t}^1_m, \dotsc, \sum_{m=1}^l \mathsf{t}^a_m \right\}^*. \]
\end{lemma}

\begin{proof}	
	Let $\mathsf{u}\in\{0,1\}^k$ and let $\mathsf{u}^-=(\mathsf{u}_1,\dotsc,\mathsf{u}_{k-1})$.
	First suppose $l<k$. Recall the set $X(l,l)$ from Definition~\ref{def:CXY}, and $\Gamma_{l,l;k}(\mathsf{u})$ from Definition~\ref{def:Gamma}.
	Consider some $x\in X(l,l)$. Since $x$ must have a fixed point, $x=(f_1,\dotsc,f_p;1)$ with $f_1,\dotsc,f_p\in P_{p^{k-1}}$, and \emph{either} $f_i,f_j\in Y(l)$ for some distinct $i,j\in[p]$ and $f_h=1$ for all $h\notin\{i,j\}$, \emph{or} $f_i\in Y(l,l)$ for a unique $i\in[p]$ and $f_j=1$ for all $j\ne i$. Thus by Lemma~\ref{lem:4.3.9},
	\begin{equation}\label{eqn:delta}
	\Gamma_{l,l;k}(\mathsf{u}) = \tbinom{p}{2}\cdot\Gamma_{l;k-1}(\mathsf{u}^-) \cdot\Gamma_{l;k-1}(\mathsf{u}^-) + p\cdot\Gamma_{l,l;k-1}(\mathsf{u}^-).
	\end{equation}
	Iterating \eqref{eqn:delta} and recalling that $\Gamma_{l,l;l}(-)=0$, we obtain
	\begin{equation}
	\Gamma_{l,l;k}(\mathsf{u}) = \binom{p}{2}\sum_{i=l}^{k-1}p^{k-1-i}\cdot\Gamma_{l;i}\big((\mathsf{u}_1,\dotsc,\mathsf{u}_i)\big)^2 = \binom{p}{2}\sum_{i=l}^{k-1}p^{k-1+i}\cdot C_l(\mathsf{u})^2 = \frac{p^k(p^k-p^l)}{2}\cdot C_l(\mathsf{u})^2,
	\end{equation}
	where the second equality follows from Lemma~\ref{lem: b=1}. 
	
	\smallskip
	
	Next, let $\phi(\underline{\mathbf{u}})\in\Lin(P_{ap^k})$. 
	We consider the value of $\phi(\underline{\mathbf{u}})\up^{\fS_{ap^k}}$ on elements of cycle type $p^lp^l$.
	Let $g\in\fS_{ap^k}$ have cycle type $p^lp^l$. Recall the set $W(l,l)$ from Definition~\ref{def:W}. Then 
	\[ \phi(\underline{\mathbf{u}})\up_{P_{ap^k}}^{\fS_{ap^k}}(g) = \tfrac{|{\bf C}_{\fS_{ap^k}}(g)|}{|P_{ap^k}|}\cdot \sum_{x\in W(l,l)}\phi(\underline{\mathbf{u}})(x). \]
	Consider some $x\in W(l,l)$. Since $P_{ap^k}\cong P_{p^k}\times\cdots\times P_{p^k}$ ($a$ times), we can write $x=x_1\cdots x_a$ with each $x_i$ in a distinct direct factor $P_{p^k}$. Moreover, \emph{either} $x_i,x_j\in X(l)$ for some distinct $i,j\in[a]$ and $x_h=1$ for all $h\notin\{i,j\}$, \emph{or} $x_i\in X(l,l)$ for a unique $i\in[a]$ and $x_j=1$ for all $j\ne i$. 
	Hence
	\begin{align}\label{eq: type pl pl}
	\sum_{x\in W(l,l)}\phi(\underline{\mathbf{u}})(x) &= \sum_{\{i,j\}\subseteq[a]}\  \sum_{x_i,x_j\in X(l)} \phi(\mathsf{u}^i)(x_i)\cdot\phi(\mathsf{u}^j)(x_j) + \sum_{i=1}^a \sum_{x_i\in X(l,l)} \phi(\mathsf{u}^i)(x_i)\nonumber\\
	&= \sum_{1\le i<j\le a} \Gamma_{l;k}(\mathsf{u}^i)\cdot \Gamma_{l;k}(\mathsf{u}^j) + \sum_{i=1}^a \Gamma_{l,l;k}(\mathsf{u}^i)\nonumber\\
	&= \sum_{1\le i<j\le a} p^k C_l(\mathsf{u}^i)\cdot p^k C_l(\mathsf{u}^j) + \sum_{i=1}^a \frac{p^k(p^k-p^l)}{2}\cdot C_l(\mathsf{u^i})^2\nonumber\\ 
	&= \frac{p^{2k}}{2}\Bigg[ \bigg(\sum_{i=1}^a C_l(\mathsf{u}^i)\bigg) \cdot \bigg( \sum_{j=1}^a C_l(\mathsf{u}^j)\bigg) - \sum_{i=1}^a C_l(\mathsf{u}^i)^2 \Bigg] + \frac{p^k(p^k-p^l)}{2}\cdot\sum_{i=1}^a C_l(\mathsf{u}^i)^2\nonumber\\
	&= \frac{p^{2k}}{2}\bigg(\sum_{i=1}^a C_l(\mathsf{u}^i)\bigg) \bigg(\sum_{j=1}^a C_l(\mathsf{u}^j)\bigg) -\frac{p^k\cdot p^l}{2}\cdot \sum_{i=1}^a C_l(\mathsf{u}^i)^2.
	\end{align}
	Since $\phi(\underline{\mathbf{s}})\up^{\fS_{ap^k}}(g')=\phi(\underline{\mathbf{t}})\up^{\fS_{ap^k}}(g')$ for $g'\in\fS_{ap^k}$ of cycle type $p^l$, we have by 
	Lemma~\ref{lem: b=1} and Theorem~\ref{thm:general-a} that 
	\[ \sum_{i=1}^a C_l(\mathsf{s}^i)=\sum_{i=1}^a C_l(\mathsf{t}^i). \]
	Substituting this into \eqref{eq: type pl pl} along with the fact that $\phi(\underline{\mathbf{s}})\up^{\fS_{ap^k}}(g)=\phi(\underline{\mathbf{t}})\up^{\fS_{ap^k}}(g)$ for $g\in\fS_{ap^k}$ of cycle type $p^lp^l$, we obtain
	\[ \sum_{i=1}^a C_l(\mathsf{s}^i)^2=\sum_{i=1}^a C_l(\mathsf{t}^i)^2. \]
	By Lemma~\ref{lem: b=1}, this gives
	\[ \sum_{i=1}^a \prod_{m=1}^l (p\mathsf{s}^i_m-1)^2 = \sum_{i=1}^a \prod_{m=1}^l (p\mathsf{t}^i_m-1)^2. \]
	In other words, we have
	\[ \sum_{i=1}^a (-q)^{\sigma_i} = \sum_{i=1}^a (-q)^{\tau_i}\quad\text{where}\quad q=p-1,\quad \sigma_i=2\sum_{m=1}^l \mathsf{s}^i_m\quad \text{and}\quad \tau_i=2\sum_{m=1}^l \mathsf{t}^i_m.\]
	Therefore by Lemma~\ref{lem: signed powers} we must have $\{\sigma_1,\dotsc,\sigma_a\}^*=\{\tau_1,\dotsc,\tau_a\}^*$ (case (ii) of Lemma~\ref{lem: signed powers} is not possible as $\sigma_i$ and $\tau_i$ are even for all $i$). The assertion of the present lemma for $l<k$ then follows directly.
	
	\smallskip
	
	Finally, suppose $l=k$. The assertion in this case follows by a similar argument to the case of $l<k$, the only difference being that we have the following instead of \eqref{eq: type pl pl} because $X(k,k)=\emptyset$:
	\begin{align*}
	\sum_{x\in W(k,k)}\phi(\underline{\mathbf{u}})(x) &= \sum_{\{i,j\}\subseteq[a]}\  \sum_{x_i,x_j\in X(k)} \phi(\mathsf{u}^i)(x_i)\cdot\phi(\mathsf{u}^j)(x_j) = \sum_{1\le i<j\le a} \Gamma_{k;k}(\mathsf{u}^i)\cdot \Gamma_{k;k}(\mathsf{u}^j)\\[5pt]
	&= \frac{p^{2k}}{2}\bigg(\sum_{i=1}^a C_k(\mathsf{u}^i)\bigg) \bigg(\sum_{j=1}^a C_k(\mathsf{u}^j)\bigg) -\frac{p^{2k}}{2}\cdot \sum_{i=1}^a C_k(\mathsf{u}^i)^2.
	\end{align*}
\end{proof}

We are now ready to translate the equations from Corollary~\ref{cor:4.18} into statements concerning multisets of (sums of) partial sums of the sequences  $\mathsf{s}^1,\dotsc,\mathsf{s}^a$ and $\mathsf{t}^1,\dotsc,\mathsf{t}^a$, as promised. 

\begin{proposition}\label{prop: claim 8}
	Let $a\in[p-1]$. Let $\phi(\underline{\mathbf{s}})$, $\phi(\underline{\mathbf{t}})\in\Lin(P_{ap^k})$ and suppose $\phi(\underline{\mathbf{s}})\up^{\fS_{ap^k}}=\phi(\underline{\mathbf{t}})\up^{\fS_{ap^k}}$. Let $b\in[a]$ and let $l_1,\dotsc,l_b$ be distinct integers in $[k]$. Then $\{\sigma_1,\dotsc,\sigma_a\}^* = \{\tau_1,\dotsc,\tau_a\}^*$ where
	\[ \sigma_j = \sum_{i=1}^b\sum_{m=1}^{l_i} \mathsf{s}^j_m\quad\mathrm{and}\quad \tau_j = \sum_{i=1}^b\sum_{m=1}^{l_i} \mathsf{t}^j_m\]
	for each $j\in[a]$.
\end{proposition}

\begin{proof}
By Corollary~\ref{cor:4.18} and Definition~\ref{def:CXY},
$$\sum_{j=1}^a \prod_{i=1}^b \prod_{m=1}^{l_i}(p\mathsf{s}^j_m-1) = \sum_{j=1}^a \prod_{i=1}^b \prod_{m=1}^{l_i}(p\mathsf{t}^j_m-1),$$
and hence $\sum_{j=1}^a (-p+1)^{\sigma_j} = \sum_{j=1}^a (-p+1)^{\tau_j}$. The assertion follows immediately if $a=1$, and if $a>1$ (in which case $p$ is therefore necessarily odd) then it follows from Lemma~\ref{lem: signed powers}: case (ii) of Lemma~\ref{lem: signed powers} cannot occur because $\sum_{j=1}^a \sigma_j=\sum_{j=1}^a \tau_j$ by Lemma~\ref{lem: claim 7}.
\end{proof}

While the conclusion of Theorem~\ref{thm: apk restated} does not follow instantly from Proposition~\ref{prop: claim 8}, only two more combinatorial arguments are needed to deduce that $\underline{\mathbf{s}}$ is a permutation of $\underline{\mathbf{t}}$. 
Before this, we give a definition which will simplify the notation and an example to illustrate this.

\begin{definition}\label{def:f}
	Let $b,k\in\mathbb{N}$. Given natural numbers $l_1,\dotsc,l_b\le k$ and a sequence $\mathsf{s}=(\mathsf{s}_1,\dotsc,\mathsf{s}_k)\in\{0,1\}^k$, define
	\[ f(l_1,\dotsc,l_b;\mathsf{s}) = \sum_{i=1}^b \sum_{m=1}^{l_i} \mathsf{s}_m.\]
	Let $a\in\mathbb{N}$. Given an $a$-tuple $\underline{\mathbf{s}}=(\mathsf{s}^1,\dotsc,\mathsf{s}^a)$ where $\mathsf{s}^i\in\{0,1\}^k$ for all $i$, define $f(l_1,\dotsc,l_b;\underline{\mathbf{s}})$ to be the multiset
	\[f(l_1,\dotsc,l_b;\underline{\mathbf{s}}) = \{ f(l_1,\dotsc,l_b;\mathsf{s}^1), \dotsc, f(l_1,\dotsc,l_b;\mathsf{s}^a) \}^*.\]
\end{definition}

Thus the result of Proposition~\ref{prop: claim 8} may be restated as 
\[ f(l_1,\dotsc,l_b;\underline{\mathbf{s}}) = f(l_1,\dotsc,l_b;\underline{\mathbf{t}}). \]

\begin{example}
	Let $k=5$ and $a=3$. Let $b=3$ with $l_1=2$, $l_2=4$ and $l_3=5$. Then
	\[ \begin{array}{ccccc}
	\mathsf{s}^1 = (0,0,0,0,1) &&&& f(2,4,5;\mathsf{s}^1) = 0 + 0 + 1 = 1\\	\mathsf{s}^2 = (1,0,1,1,0) &&\implies && f(2,4,5;\mathsf{s}^2) = 1 + 3 + 3 = 7\\
	\mathsf{s}^3 = (1,1,0,0,1) &&&& f(2,4,5;\mathsf{s}^3) = 2 + 2 + 3 = 7
	\end{array} \]
	and so $f(2,4,5;\underline{\mathbf{s}}) = \{1,7,7\}^*$.
	\hfill$\lozenge$
\end{example}

The following lemma shows that equality of two multisets is preserved when we add 1 to the same number of elements in each multiset.

\begin{lemma}\label{lem:keylemma}
Let $\{l_1,\dots,l_b\}^*=\{m_1,\dots,m_b\}^*$. Suppose that in addition we have $$\{l_1+1,\dots,l_c+1,l_{c+1},\dots,l_b\}^*=\{m_1+1,\dots,m_c+1,m_{c+1},\dots,m_b\}^*$$
for some $c\in[b-1]$. Then $\{l_1,\dots,l_c\}^*=\{m_1,\dots,m_c\}^*.$
\end{lemma}
\begin{proof}
Suppose for the sake of contradiction that $\{l_1,\dots,l_c\}^*\ne\{m_1,\dots,m_c\}^*.$ Let us suppose without loss of generality that $l_1\le \dots \le l_c$ and $m_1\le\dots\le m_c$. Let $j\le c$ be maximal such that $l_j\neq m_j$, and without loss of generality we assume that $l_j<m_j$.

Given a multiset $S$ and $v\in\mathbb{N}$ we let $S_{\ge v}=\{x: x\in S,\ x\ge v\}^*$. Observe that since $\{l_1,\dots,l_b\}^*=\{m_1,\dots,m_b\}^*$ we have that $\{l_1,\dots,l_b\}_{\ge v}^*=\{m_1,\dots,m_b\}_{\ge v}^*$ for any $v$. Similarly 
$$\{l_1+1,\dots,l_c+1,l_{c+1},\dots,l_b\}_{\ge v}^*=\{m_1+1,\dots,m_c+1,m_{c+1},\dots,m_b\}_{\ge v}^*$$
for any $v$. 
We now consider $v=m_j+1$. Note that 
$$|\{l_1+1,\dots,l_c+1,l_{c+1},\dots,l_b\}_{\ge m_j+1}^*|-|\{l_1,\dots,l_b\}_{\ge m_j+1}^*|$$
$$=|\{l_{j+1}+1,\dots,l_c+1\}_{\ge m_j+1}^*|-|\{l_{j+1},\dots,l_c\}_{\ge m_j+1}^*|$$
by cancelling off equal elements in the two multisets and noting that $l_i<m_j$ for all $i\le j$. 
Moreover, 
\begin{align*}
&\quad\ |\{m_1+1,\dots,m_c+1,m_{c+1},\dots,m_b\}_{\ge m_j+1}^*|-|\{m_1,\dots,m_b\}_{\ge m_j+1}^*|\\
&\ge |\{m_{j}+1,\dots,m_c+1\}_{\ge m_j+1}^*|-|\{m_j,\dots,m_c\}_{\ge m_j+1}^*|\\ 
&= 1+|\{m_{j+1}+1,\dots,m_c+1\}_{\ge m_j+1}^*|-|\{m_{j+1},\dots,m_c\}_{\ge m_j+1}^*|\\
&=1+|\{l_{j+1}+1,\dots,l_c+1\}_{\ge m_j+1}^*|-|\{l_{j+1},\dots,l_c\}_{\ge m_j+1}^*|
\end{align*}
by maximality of $j$. In particular, we have
$$|\{m_1+1,\dots,m_c+1,m_{c+1},\dots,m_b\}_{\ge m_j+1}^*|-|\{m_1,\dots,m_b\}_{\ge m_j+1}^*|$$
$$\ge 1+|\{l_1+1,\dots,l_c+1,l_{c+1},\dots,l_b\}_{\ge m_j+1}^*|-|\{l_1,\dots,l_b\}_{\ge m_j+1}^*|$$
which is a contradiction. 
\end{proof}

We are finally ready to deduce Theorem~\ref{thm: apk restated} from the equalities of the functions $f$ given by Proposition~\ref{prop: claim 8} and Definition~\ref{def:f}.

\begin{theorem}\label{thm: part2}
Let $a,k\in\mathbb{N}$. Let $\underline{\mathbf{s}}=(\mathsf{s}^1,\dotsc,\mathsf{s}^a)$ and $\underline{\mathbf{t}}=(\mathsf{t}^1,\dotsc,\mathsf{t}^a)$ where $\mathsf{s}^i,\mathsf{t}^i\in\{0,1\}^k$ for all $i$. Suppose that for any distinct integers $l_1,l_2,\dotsc,l_b\in[k]$ such that $b\in[a]$ we have
\begin{equation}\label{eq: equal f sums}
f(l_1,\dotsc,l_b;\underline{\mathbf{s}}) = f(l_1,\dotsc,l_b;\underline{\mathbf{t}}).
\end{equation}
Then there exists a permutation $\sigma\in\fS_a$ such that $\mathsf{s}^i = \mathsf{t}^{\sigma(i)}$ for all $i$.
\end{theorem}
\begin{proof}
We prove the assertion for $(a,k)$ by induction on $a+k$. When $k=1$ and $a$ is arbitrary, the assertion is clear since the single term of each sequence $\mathsf{s}^i$ is simply $f(1;\mathsf{s}^i)$. When $a=1$ and $k$ is arbitrary, note that $\mathsf{s}^1_1=f(1;\mathsf{s}^1)=f(1;\mathsf{t}^1)=\mathsf{t}^1_1$, and $\mathsf{s}^1_r = f(r;\mathsf{s}^1)-f(r-1;\mathsf{s}^1) = f(r;\mathsf{t}^1)-f(r-1;\mathsf{t}^1) = \mathsf{t}^1_r$ for all $r\in\{2,3,\dotsc,k\}$, so $\mathsf{s}^1=\mathsf{t}^1$ as required. 

Now suppose $a,k\ge 2$. Write $\hat{\mathsf{s}}^i$ for the sequence $(\mathsf{s}^i_j)_{j=2}^k\in\{0,1\}^{k-1}$ and let $\underline{\hat{\mathsf{s}}} = (\hat{\mathsf{s}}^1,\dotsc,\hat{\mathsf{s}}^a)$. Define $\hat{\mathsf{t}}^i$ and $\underline{\hat{\mathsf{t}}}$ similarly. 

First suppose that $\mathsf{s}^i_1=\mathsf{t}^i_1=z$ for all $i\in[a]$, for some $z\in\{0,1\}$. Observe that for any distinct integers $l_1,\dotsc,l_b\in[k-1]$ such that $b\in[a]$ we have
\[ f(l_1+1,\dotsc,l_b+1;\mathsf{s}^i) = f(l_1,\dotsc,l_b;\hat{\mathsf{s}}^i) + bz, \]
and similarly
\[ f(l_1+1,\dotsc,l_b+1;\mathsf{t}^i) = f(l_1,\dotsc,l_b;\hat{\mathsf{t}}^i) + bz. \]
Thus by (\ref{eq: equal f sums}), we have that
\[ f(l_1,\dotsc,l_b;\underline{\hat{\mathsf{s}}}) = f(l_1,\dotsc,l_b;\underline{\hat{\mathsf{t}}}).\] 
By the inductive hypothesis for $(a,k-1)$, there exists a permutation $\sigma\in\operatorname{Sym}[a]$ such that $\hat{\mathsf{s}}^i = \hat{\mathsf{t}}^{\sigma(i)}$ for all $i\in[a]$. Therefore $\mathsf{s}^i = \mathsf{t}^{\sigma(i)}$ as required.

Otherwise, we may now suppose that not all $\mathsf{s}^i_1$ and $\mathsf{t}^i_1$ are equal. Let $I_s=\{i\in[a]: \mathsf{s}^i_1=1\}$ and define $I_t$ similarly. Since $f(1;\underline{\mathbf{s}})=f(1;\underline{\mathbf{t}})$, then $|I_s|=|I_t|$, so we may without loss of generality reorder $\underline{\mathbf{t}}$ to assume that $\mathsf{s}^i_1=\mathsf{t}^i_1$ for all $i\in[a]$. Let $I=I_s=I_t$ and note that $|I|\in[a-1]$.

For any distinct integers $l_1,\dotsc,l_b\in[k-1]$ such that $b\in[a-1]$ we have
\[ f(l_1+1,\dotsc,l_b+1;\mathsf{s}^i) = b\mathsf{s}^i_1 + f(l_1,\dotsc,l_b;\hat{\mathsf{s}}^i) \]
and
\[ f(l_1+1,\dotsc,l_b+1;\mathsf{t}^i) = b\mathsf{s}^i_1 + f(l_1,\dotsc,l_b;\hat{\mathsf{t}}^i) \]
since $\mathsf{s}^i_1=\mathsf{t}^i_1$. Thus
\begin{equation}\label{eq:1}
\begin{split}
f(l_1+1,&\dotsc,l_b+1;\underline{\mathbf{s}}) = \{b\mathsf{s}^i_1+ f(l_1,\dotsc,l_b;\hat{\mathsf{s}}^i)\ :\ i\in[a] \}^*\\
&= f(l_1+1,\dotsc,l_b+1;\underline{\mathbf{t}}) = \{b\mathsf{s}^i_1+ f(l_1,\dotsc,l_b;\hat{\mathsf{t}}^i)\ :\ i\in[a] \}^*.
 \end{split}
\end{equation}
In addition,
\begin{equation*}
\begin{split}
f(1, &l_1+1,\dotsc,l_b+1;\underline{\mathbf{s}}) = \{(b+1)\mathsf{s}^i_1 + f(l_1,\dotsc,l_b;\hat{\mathsf{s}}^i)\ :\ i\in[a] \}^*\\
&=f(1,l_1+1,\dotsc,l_b+1;\underline{\mathbf{t}}) = \{(b+1)\mathsf{s}^i_1 + f(l_1,\dotsc,l_b;\hat{\mathsf{t}}^i)\ :\ i\in[a] \}^*.
\end{split}
\end{equation*}
Therefore, by Lemma~\ref{lem:keylemma}, we have that
$$\{b+f(l_1,\dotsc,l_b;\hat{\mathsf{s}}^i)\ :\ i\in I \}^* = \{b+f(l_1,\dotsc,l_b;\hat{\mathsf{t}}^i)\ :\ i\in I \}^*,$$
which implies 
\begin{equation}\label{eq:2}
\{f(l_1,\dotsc,l_b;\hat{\mathsf{s}}^i)\ :\ i\in I \}^* = \{f(l_1,\dotsc,l_b;\hat{\mathsf{t}}^i)\ :\ i\in I \}^*
\end{equation}
and therefore by (\ref{eq:1}) also
\begin{equation}\label{eq:3}
\{f(l_1,\dotsc,l_b;\hat{\mathsf{s}}^i)\ :\ i\in [a]\setminus I \}^* = \{f(l_1,\dotsc,l_b;\hat{\mathsf{t}}^i)\ :\ i\in [a]\setminus I \}^*.
\end{equation}
Let $\hat{\underline{\mathbf{s}}}^{(1)}$ and $\hat{\underline{\mathbf{s}}}^{(0)}$ be the sequences $(\hat{\mathsf{s}}^i)_{i\in I}$ and $(\hat{\mathsf{s}}^i)_{i\notin I}$ respectively, and define $\hat{\underline{\mathbf{t}}}^{(1)}$ and $\hat{\underline{\mathbf{t}}}^{(0)}$ similarly. Since~(\ref{eq:2}) and~(\ref{eq:3}) hold for any valid choice of $\{l_i\}$, this tells us that for any distinct integers $l_1,\dotsc,l_b\in[k-1]$ such that $b\in[a-1]$ we have
$$f(l_1,\dotsc,l_b;\hat{\underline{\mathbf{s}}}^{(1)}) = f(l_1,\dotsc,l_b;\hat{\underline{\mathbf{t}}}^{(1)})\quad\mathrm{and}\quad f(l_1,\dotsc,l_b;\hat{\underline{\mathbf{s}}}^{(0)}) = f(l_1,\dotsc,l_b;\hat{\underline{\mathbf{t}}}^{(0)}).$$
Since $|I|$ and $a-|I|$ are both at most $a-1$, we may apply the inductive hypotheses for $(|I|,k-1)$ and $(a-|I|,k-1)$ to obtain permutations $\sigma_1\in\operatorname{Sym}I$ such that $\hat{\mathsf{s}}^i=\hat{\mathsf{t}}^{\sigma_1(i)}$ for all $i\in I$ (and hence $\mathsf{s}^i=\mathsf{t}^{\sigma_1(i)}$) and $\sigma_0\in\operatorname{Sym}([a]\setminus I)$ such that $\hat{\mathsf{s}}^i=\hat{\mathsf{t}}^{\sigma_0(i)}$ for all $i\notin I$ (and hence $\mathsf{s}^i=\mathsf{t}^{\sigma_0(i)}$). Finally, let $\sigma=\sigma_0\cdot\sigma_1\in\operatorname{Sym}[a]$, so $\mathsf{s}^i=\mathsf{t}^{\sigma(i)}$ for all $i$ as desired.
\end{proof}

\begin{proof}[Proof of Theorem~\ref{thm: apk restated}]
This follows from Proposition~\ref{prop: claim 8} and Theorem~\ref{thm: part2}.
\end{proof}

\begin{proof}[Proof of Theorem~\ref{thm: apk}]
	This is equivalent to Theorem~\ref{thm: apk restated}.
\end{proof}

\bigskip


\appendix
\section{Examples for Proposition~\ref{prop: b>1}}\label{sec:appendix}

We follow the notation in the statement of Proposition~\ref{prop: b>1}.

\begin{example}\label{ex:gamma b=2}
	In this example, we prove that Proposition~\ref{prop: b>1} holds when $b=2$, i.e.~that
	\[ \Gamma_{l_1,l_2;k}(\mathsf{u}) = p^k\cdot C_{l_1}(\mathsf{u})\cdot C_{l_2}(\mathsf{u})\cdot (p^k-p^{l_2}). \]

	First, we have by combining Definitions~\ref{def:Gamma} and~\ref{def:CXY} that
	\begin{equation}\label{eqn:b=2 initial}
	\tag{\ref{ex:gamma b=2}(a)}
	\Gamma_{l_1,l_2;k}(\mathsf{u}) = \sum_{x\in X(l_1,l_2)}\phi(\mathsf{u})(x).
	\end{equation}
	Consider some $x\in X(l_1,l_2)$. Since $l_1<l_2<k$, then $x$ must have a fixed point. Therefore $x=(f_1,\dotsc,f_p;1)$ with $f_1,\dotsc,f_p\in P_{p^{k-1}}$, and
	\begin{itemize}
		\item[(i)] \emph{either} $f_i\in Y(l_1)$, $f_j\in Y(l_2)$ for some distinct $i,j\in[p]$ and $f_h=1$ for all $h\notin\{i,j\}$,
		\item[(ii)] \emph{or} $f_i\in Y(l_1,l_2)$ for a unique $i\in[p]$ and $f_j=1$ for all $j\ne i$.
	\end{itemize}
	By Lemma~\ref{lem:4.3.9}, in case (i) we have $\phi(\mathsf{u})(x) = \phi(\mathsf{u}^-)(f_i)\cdot \phi(\mathsf{u}^-)(f_j)$, while in case (ii) we have $\phi(\mathsf{u})(x) = \phi(\mathsf{u}^-)(f_i)$.
	Since there are $p(p-1)$ choices for the ordered pair $(i,j)$ in case (i) and $p$ choices for $i$ in case (ii),
	\begin{align}\label{eqn:b=2 cycle types}
	\tag{\ref{ex:gamma b=2}(b)}
	\Gamma_{l_1,l_2;k}(\mathsf{u}) 
	&= p(p-1)\cdot\Gamma_{l_1;k-1}(\mathsf{u}^-)\cdot\Gamma_{l_2;k-1}(\mathsf{u}^-) + p\cdot\Gamma_{l_1,l_2;k-1}(\mathsf{u}^-).
	\end{align}
	Iterating \eqref{eqn:b=2 cycle types} and recalling that $\Gamma_{l_1,l_2;l_2}(-)=0$, we therefore have that
	\begin{equation}\label{eqn:b=2c}
	\tag{\ref{ex:gamma b=2}(c)}
	\Gamma_{l_1,l_2;k}(\mathsf{u}) = \sum_{i=l_2}^{k-1} p^{k-i}(p-1) \cdot\Gamma_{l_1;i}\big((\mathsf{u}_1,\dotsc,\mathsf{u}_i)\big)\cdot\Gamma_{l_2;i}\big((\mathsf{u}_1,\dotsc,\mathsf{u}_i)\big),
	\end{equation}
	From Lemma~\ref{lem: b=1}, we know that $\Gamma_{l_j;i}\big((\mathsf{u}_1,\dotsc,\mathsf{u}_i)\big) = p^iC_{l_j}(\mathsf{u})$ for any $j$, and so
	\begin{equation}\label{eqn:b=2d}
	\tag{\ref{ex:gamma b=2}(d)}
	\Gamma_{l_1,l_2;k}(\mathsf{u}) = \sum_{i=l_2}^{k-1} p^{k-i}(p-1)\cdot p^i C_{l_1}(\mathsf{u}) \cdot p^i C_{l_2}(\mathsf{u}).
	\end{equation}
	Define $\Sigma_2$ by setting $\Gamma_{l_1,l_2;k}(\mathsf{u}) = \Sigma_2\cdot p^k C_{l_1}(\mathsf{u}) C_{l_2}(\mathsf{u})$, i.e.
	\begin{equation}\label{eqn:b=2e}
	\tag{\ref{ex:gamma b=2}(e)}
	\Sigma_2:= \sum_{i=l_2}^{k-1} p^{-i}(p-1)\cdot p^i\cdot p^i.
	\end{equation}
	Since
	\begin{equation}\label{eqn:b=2}
	\tag{\ref{ex:gamma b=2}(f)}
	\Sigma_2 = \sum_{i=l_2}^{k-1} (p-1) p^i = p^k-p^{l_2},
	\end{equation}
	we obtain
	\[  \Gamma_{l_1,l_2;k}(\mathsf{u}) = \sum_{i=l_2}^{k-1} p^{k-i}(p-1)\cdot p^i C_{l_1}(\mathsf{u}) \cdot p^i C_{l_2}(\mathsf{u})
	= p^k\cdot C_{l_1}(\mathsf{u})\cdot C_{l_2}(\mathsf{u})\cdot (p^k-p^{l_2}). \]
	This concludes the case of $b=2$.\hfill$\lozenge$
\end{example}

\begin{example}\label{ex:gamma b=3}
	We prove that Proposition~\ref{prop: b>1} holds when $b=3$, i.e.~that
	\[ \Gamma_{l_1,l_2,l_3;k}(\mathsf{u}) = p^k\cdot C_{l_1}(\mathsf{u})\cdot C_{l_2}(\mathsf{u})\cdot C_{l_3}(\mathsf{u})\cdot(p^k-p^{l_3})(p^k-p^{l_3}-p^{l_2}), \]
	using the case of $b=2$. We thereby illustrate a specific example of the inductive step in the proof of Proposition~\ref{prop: b>1}. In the calculations below, we will associate certain terms with partitions $\nu$ of the set $[b]=\{1,2,3\}$. This is to help illustrate the analogy between the case of $b=3$ and the general proof for all $b\ge 3$ in Section~\ref{sec:char values}. Moreover, equations (\ref{ex:gamma b=3}(a)--(f)) below are labelled as such to illustrate their correspondence to equations (\ref{ex:gamma b=2}(a)--(f)) above.
	
	\smallskip
	
	From Definitions~\ref{def:Gamma} and~\ref{def:CXY}, we have
	\begin{equation}\label{eqn:eg1}
	\tag{\ref{ex:gamma b=3}(a)}
	\Gamma_{l_1,l_2,l_3;k}(\mathsf{u}) = \sum_{x\in X(l_1,l_2,l_3)}\phi(\mathsf{u})(x).
	\end{equation}
	Consider some $x\in X(l_1,l_2,l_3)$. 
	Since $x$ has a fixed point, $x$ must be of the form $x=(f_1,\dots,f_p;1)$ with $f_1,\dotsc,f_p\in P_{p^{k-1}}$, and so
	\[ \phi(\mathsf{u})(x) = \phi(\mathsf{u}^-)(f_1)\cdot \phi(\mathsf{u}^-)(f_2)\cdots\phi(\mathsf{u}^-)(f_p) \]
	by Lemma~\ref{lem:4.3.9}.
	Moreover, the product of the cycle types of $f_i$ over all $i$ equals $p^{l_1}p^{l_2}p^{l_3}$ since $x\in X(l_1,l_2,l_3)$. 
	We can thus associate to $x$ a partition $\nu$ of the set $\{1,2,3\}$ according to which $p^{l_w}$ ($w\in\{1,2,3\}$) occur in the cycle type of $f_1,\dotsc,f_p$ as follows.
	
	For each $i$, say $f_i\in Y(\{l_w : w\in \nu_i \})$ for some subset $\nu_i\subseteq\{1,2,3\}$, and $\nu_1,\dotsc,\nu_p$ are disjoint with $\nu_1\cup\cdots\cup\nu_p=\{1,2,3\}$. (If $\nu_i=\emptyset$ then $f_i=1$.) In other words, $\nu=\{\nu_i : \nu_i\ne\emptyset\}\in\partition{\{1,2,3\}}$ describes how the cycle lengths $p^{l_w}$ are grouped together to give the cycle types of $f_1,\dotsc,f_p$.
	Given $\nu$, there are $p(p-1)\cdots(p-|\nu|+1)$ many ways of assigning the elements of $\nu$ to $f_1,\dotsc,f_p$.
	
	Hence, by grouping $x\in X(l_1,l_2,l_3)$ according to its associated $\nu\in\partition\{1,2,3\}$, \eqref{eqn:eg1} may be rewritten as follows:
	\begin{small}
		\begin{equation}\label{eqn:eg2}
		\tag{\ref{ex:gamma b=3}(b)}
		\begin{array}{rll}
		\Gamma_{l_1,l_2,l_3;k}(\mathsf{u}) 
		%
		%
		= & p\cdot\Gamma_{l_1,l_2,l_3;k-1}(\mathsf{u}^-) & \quad\nu=\big\{\{1,2,3\}\big\}\\[4pt]
		& +p(p-1)\cdot\Gamma_{l_1,l_2;k-1}(\mathsf{u}^-)\cdot \Gamma_{l_3;k-1}(\mathsf{u}^-) & \quad\nu=\big\{ \{1,2\},\{3\} \big\}\\[4pt]
		& +p(p-1)\cdot\Gamma_{l_1,l_3;k-1}(\mathsf{u}^-)\cdot \Gamma_{l_2;k-1}(\mathsf{u}^-) & \quad\nu=\big\{ \{1,3\},\{2\} \big\}\\[4pt]
		& +p(p-1)\cdot\Gamma_{l_2,l_3;k-1}(\mathsf{u}^-)\cdot \Gamma_{l_1;k-1}(\mathsf{u}^-) & \quad\nu=\big\{ \{2,3\},\{1\} \big\}\\[4pt]
		& +p(p-1)(p-2)\cdot \Gamma_{l_1;k-1}(\mathsf{u}^-)\cdot \Gamma_{l_2;k-1}(\mathsf{u}^-)\cdot \Gamma_{l_3;k-1}(\mathsf{u}^-) & \quad\nu=\big\{ \{1\},\{2\},\{3\} \big\}
		\end{array}
		\end{equation}
	\end{small}

	Let $\mathsf{v}_i=(\mathsf{u}_1,\dotsc,\mathsf{u}_i)$. Iterating \eqref{eqn:eg2} and recalling that $\Gamma_{l_1,l_2,l_3;l_3}(-)=0$, we therefore have that
	\begin{small}
		\begin{equation}\label{eqn:eg3}
		\tag{\ref{ex:gamma b=3}(c)}
		\begin{array}{lll}
		\Gamma_{l_1,l_2,l_3;k}(\mathsf{u}) = \sum_{i=l_3}^{k-1} p^{k-i} [ & (p-1)\cdot \Gamma_{l_1,l_2;i}(\mathsf{v}_i)\cdot \Gamma_{l_3;i}(\mathsf{v}_i) & \quad \nu=\big\{\{1,2\},\{3\}\big\} \\[4pt]
		& + (p-1)\cdot \Gamma_{l_1,l_3;i}(\mathsf{v}_i)\cdot \Gamma_{l_2;i}(\mathsf{v}_i) & \quad \nu=\big\{\{1,3\},\{2\}\big\} \\[4pt]
		& + (p-1)\cdot \Gamma_{l_2,l_3;i}(\mathsf{v}_i)\cdot \Gamma_{l_1;i}(\mathsf{v}_i) & \quad \nu=\big\{\{2,3\},\{1\}\big\} \\[4pt]
		& + (p-1)(p-2)\cdot \Gamma_{l_1;i}(\mathsf{v}_i)\cdot \Gamma_{l_2;i}(\mathsf{v}_i)\cdot  \Gamma_{l_3;i}(\mathsf{v}_i)\ ] & \quad \nu=\big\{\{1\}, \{2\}, \{3\}\big\}
		\end{array}
		\end{equation}
	\end{small}

	We use Lemma~\ref{lem: b=1} and the $b=2$ case of Proposition~\ref{prop: b>1} to deduce that
	\begin{small}
		\begin{equation*}
		\begin{array}{rll}
		\Gamma_{l_1,l_2,l_3;k}(\mathsf{u}) = \sum_{i=l_3}^{k-1} p^{k-i} \big[ & \multicolumn{2}{l}{ (p-1)\cdot p^i\ C_{l_1}(\mathsf{u})\ C_{l_2}(\mathsf{u})\ (p^i-p^{l_2})\cdot p^i\ C_{l_3}(\mathsf{u}) }\\[4pt]
		& \multicolumn{2}{l}{ + (p-1)\cdot p^i\ C_{l_1}(\mathsf{u})\ C_{l_3}(\mathsf{u})\ (p^i-p^{l_3})\cdot p^i\ C_{l_2}(\mathsf{u}) }\\[4pt]
		& \multicolumn{2}{l}{ + (p-1)\cdot p^i\ C_{l_2}(\mathsf{u})\ C_{l_3}(\mathsf{u})\ (p^i-p^{l_3})\cdot p^i\ C_{l_1}(\mathsf{u}) }\\[4pt]
		& \multicolumn{2}{l}{ + (p-1)(p-2)\cdot p^i\ C_{l_1}(\mathsf{u})\cdot p^i\ C_{l_2}(\mathsf{u})\cdot p^i\ C_{l_3}(\mathsf{u})\quad \big] }
		\end{array}
		\end{equation*}
		\begin{equation}\label{eqn:eg4}
		\tag{\ref{ex:gamma b=3}(d)}
		\begin{array}{rll}		
		= \sum_{i=l_3}^{k-1} p^{k-i}\ \cdot C_{l_1}(\mathsf{u})\ C_{l_2}(\mathsf{u})\ C_{l_3}(\mathsf{u}) \big[ & (p-1)\cdot p^{2i}\ (p^i-p^{l_2})  & \qquad \nu=\big\{\{1,2\},\{3\}\big\} \\[4pt]
		& +(p-1)\cdot p^{2i}\ (p^i-p^{l_3}) & \qquad \nu=\big\{\{1,3\},\{2\}\big\} \\[4pt]
		& +(p-1)\cdot p^{2i}\ (p^i-p^{l_3}) & \qquad \nu=\big\{\{2,3\},\{1\}\big\} \\[4pt]
		& +(p-1)(p-2)\cdot p^{3i}\qquad \big] & \qquad \nu=\big\{ \{1\},\{2\},\{3\} \big\}
		\end{array}
		\end{equation}
	\end{small}

	Thus, it remains to prove that $\Sigma_3 = (p^k-p^{l_3})(p^k-p^{l_3}-p^{l_2})$, where we define $\Sigma_3$ as
	\begin{small}
		\begin{equation}\label{eqn:eg5}
		\tag{\ref{ex:gamma b=3}(e)}
		\Sigma_3:= \sum_{i=l_3}^{k-1} p^{-i} \big[ (p-1)p^{2i}(p^i-p^{l_2}) + (p-1)p^{2i}(p^i-p^{l_3}) + (p-1)p^{2i}(p^i-p^{l_3}) + (p-1)(p-2)p^{3i} \big]
		\end{equation}
	\end{small}
	
	Let $Q_i=(p-1)p^i$, and recall that $\sum_{i=l_3}^{k-1} Q_i = p^k-p^{l_3}$ (see~\eqref{eqn:b=2} from the case of $b=2$). Note further that
	\begin{equation}\label{eqn:eg6}
	\tag{\ref{ex:gamma b=3}(f)}
	\sum_{j=l_3}^{i-1} Q_j = p^i-p^{l_3} \quad\text{for all }i\in\mathbb{N}\text{ such that }i>l_3.
	\end{equation}
	We group the terms appearing in $\Sigma_3$ by their associated $\nu\in\partition{\{1,2,3\}}$ according to the partition $\gamma\in\partition{\{2,3\}}$ obtained from $\nu$ by removing 1. That is, we group the terms as follows:
	\begin{itemize}
		\item[(i)] $\gamma = \big\{\{2\},\{3\} \big\}$ : $\nu=\big\{ \{1,2\},\{3\} \big\},\ \nu=\big\{\{1,3\},\{2\} \big\},\ \nu=\big\{\{1\},\{2\},\{3\} \big\}$, and
		\item[(ii)] $\gamma = \big\{ \{2,3\} \big\}$ : $\nu=\big\{ \{2,3\},\{1\} \big\}$.
	\end{itemize}
	By first summing together the terms corresponding to the same $\gamma$, we obtain
	\begin{small}
		\begin{align}\label{eqn:eg7}
		\Sigma_3 &= \sum_{i=l_3}^{k-1} \big[ \underbrace{(p-1) p^i (p^{i+1}-p^{l_3}-p^{l_2})}_{\gamma = \{\{2\},\{3\} \} } + \underbrace{(p-1) p^i (p^i-p^{l_3})}_{\gamma = \{ \{2,3\} \}} \big]\nonumber\\
		&= \sum_{i=l_3}^{k-1}\big[ Q_i\cdot (p^{i+1}-p^{l_3}-p^{l_2}) + (p-1) p^i (p^i-p^{l_3}) \big]\nonumber\\
		&= (-p^{l_3}-p^{l_2})\sum_{i=l_3}^{k-1} Q_i + \sum_{i=l_3}^{k-1} \big[ p^{i+1} Q_i + (p-1) p^i (p^i-p^{l_3}) \big]\nonumber\\
		&= (-p^{l_3}-p^{l_2})(p^k-p^{l_3}) + \sum_{i=l_3}^{k-1} \left[ p^{i+1} Q_i + (p-1) p^i \sum_{j=l_3}^{i-1} Q_j \right]\nonumber\\
		&= (-p^{l_3}-p^{l_2})(p^k-p^{l_3}) + \sum_{h=l_3}^{k-2} Q_h\cdot\left( p^{h+1} + (p-1)\sum_{z=h+1}^{k-1}p^z \right) + p^k Q_{k-1}\nonumber\\
		&= (-p^{l_3}-p^{l_2})(p^k-p^{l_3}) + p^k \sum_{h=l_3}^{k-1} Q_h\nonumber\\
		&= (-p^{l_3}-p^{l_2})(p^k-p^{l_3}) + p^k(p^k-p^{l_3}) = (p^k-p^{l_3})(p^k-p^{l_3}-p^{l_2}),\nonumber\tag{\ref{ex:gamma b=3}(g)}
		\end{align}
	\end{small}
	as desired.
	\hfill$\lozenge$
\end{example}

\bigskip


\end{document}